\documentclass[11pt]{amsart}

\usepackage[authoryear]{natbib}

\RequirePackage[OT1]{fontenc}
\RequirePackage{amsthm,amsmath,amssymb}
\usepackage[colorlinks=true, urlcolor=blue, linkcolor=black, citecolor=black, pdftex]{hyperref}
\usepackage{graphicx}
\usepackage[usenames]{color}
\usepackage{csquotes}
\usepackage[utf8]{inputenc}
\usepackage{amsmath}
\usepackage{graphicx}
\usepackage{amsfonts}
\usepackage{amssymb}
\usepackage{amsthm}
\usepackage{enumerate}
\usepackage{mathrsfs} 
\usepackage{hyperref} 
\usepackage[english]{babel}
\usepackage{titlesec}
\usepackage{lscape}
\usepackage{bbm}
\usepackage{tikz}

\usepackage{algorithm}
\makeatletter
\titleformat{\section}[block]{\large \scshape \filcenter}{\thesection.}{0.5em}{}
\titleformat{\subsection}[runin]{\normalfont \bfseries}{\thesubsection.}{0.5 em}{}
\titleformat{\subsubsection}[runin]{\normalfont \bfseries}{\thesubsubsection.}{0.5 em}{}
\usepackage[authoryear]{natbib}

\makeatother

\newtheorem{thm}{Theorem}
\newtheorem{prop}[thm]{Proposition}
\newtheorem{lemme}{Lemma}

\newtheorem{cor}{Corollary}
\newtheorem{hyp}{Assumption}

\newtheorem{Claim}{Claim}
\newtheorem{exemple}{Example}

\renewcommand{\d}{\, \mathrm{d}}
\newcommand{\R}{\mathbb{R}}

\renewcommand{\P}{\mathbb{P}\,}
\newcommand{\N}{\mathbb{N}}

\newcommand{\E}{\mathbb{E}}

\newcommand{\V}{\mathbb{V}}

\newcommand{\A}{\mathcal{A}}

\newcommand{\F}{\mathscr{F}}
\newcommand{\EE}{\mathcal{E}}
\renewcommand{\L}{\mathbb{L}}
\newcommand{\M}{\mathcal{M}}

\newcommand{\var}{\mathrm{ var }}
\newcommand{\pen}{\mathrm{ pen }}

\newcommand{\I}{{\mathbf{I}}}

\newcommand{\Z}{{\mathbb{Z}}}

\newcommand{\W}{{\mathbb{W}}}

\newcommand{\LL}{\mathbb{L}}

\newcommand{\XX}{\mathbb{X}}
\newcommand{\TT}{\mathbb{T}}

\newcommand{\1}{\mathbbm{1}}

\def\restriction#1#2{\mathchoice
              {\setbox1\hbox{${\displaystyle #1}_{\scriptstyle #2}$}
              \restrictionaux{#1}{#2}}
              {\setbox1\hbox{${\textstyle #1}_{\scriptstyle #2}$}
              \restrictionaux{#1}{#2}}
              {\setbox1\hbox{${\scriptstyle #1}_{\scriptscriptstyle #2}$}
              \restrictionaux{#1}{#2}}
              {\setbox1\hbox{${\scriptscriptstyle #1}_{\scriptscriptstyle #2}$}
              \restrictionaux{#1}{#2}}}
\def\restrictionaux#1#2{{#1\,\smash{\vrule height .8\ht1 depth .85\dp1}}_{\,#2}} 
\usepackage{verbatim}

\makeatletter
\def\hlinewd#1{%
\noalign{\ifnum0=`}\fi\hrule \@height #1 %
\futurelet\reserved@a\@xhline}
\makeatother

\title{Estimation of the transition density of a Markov chain}
\date{October, 2012} 
\author{Mathieu Sart} 
\address{Université de Nice Sophia-Antipolis, Laboratoire J-A Dieudonné, Parc Valrose, 06108 Nice
cedex 02, France.}
\email{msart@unice.fr}

\begin{document}
\maketitle
\begin{abstract}
We present two data-driven procedures to estimate the transition density of an homogeneous Markov chain. The first yields to a piecewise constant estimator on a suitable random partition. 
By using an Hellinger-type loss, we establish non-asymptotic risk bounds for our estimator when the square root of the transition density belongs to possibly inhomogeneous Besov spaces  with possibly small regularity index. Some simulations are also provided. The second procedure is of theoretical interest and leads to a general model selection theorem from which we derive rates of convergence over a very wide range of possibly inhomogeneous and anisotropic Besov spaces. We also investigate the rates that can be achieved under structural assumptions on the transition density.
\end{abstract}

\section{Introduction.}
Consider a time-homogeneous Markov chain $(X_i)_{i \in \N}$ defined on an abstract probability space $(\Omega, \EE, \P)$ with values in the measured space $(\XX, \mathcal{F}, \mu)$. We assume that   for each $x \in \XX$, the conditional law $\mathcal{L} (X_{i+1} \mid X_i = x)$ admits a density $s (x, \cdot)$ with respect to~$\mu$. 
Our aim is to estimate the transition density $(x,y) \mapsto s (x,y)$ on a subset $A = A_1 \times A_2$ of $\XX^2$ from the observations $X_0,\dots,X_{n}$.

Many papers are devoted to this statistical setting. A popular method to build an estimator of $s$ is to   divide an estimator of the joint density of $(X_i,X_{i+1})$ by an estimator of the density of $X_i$.  The resulting estimator is called a quotient estimator.
\cite{Roussas1969},~\cite{Athreya1998} considered Kernel estimators for the densities of $X_i$ and $(X_i,X_{i+1})$. They proved consistence  and asymptotic normality of the quotient estimator. Other properties of this estimator were established:~\cite{Roussas1991},~\cite{Dorea2002} showed strong consistency,~\cite{Basu1998} proved a Berry-Essen type theorem and \cite{Doukhan1983} bounded from above the integrated quadratic risk under Sobolev constraints.
\cite{Clemenccon2000} investigated  the minimax rates when  $A = [0,1]^2$, $\XX^2 = \R^2$. Given two smoothness classes $\F_1$ and~$\F_2$ of real valued functions on~$[0,1]^2$ and~$[0,1]$ respectively (balls of Besov spaces), he  established the lower bounds over the class 
  $$\mathscr{F} = \left\{ \varphi, \; \forall x,y \in [0,1],   \varphi (x,y) = \frac{\varphi_1(x,y)}{ \varphi_2(x)} , \, (\varphi_1,\varphi_2) \in \F_1 \times \F_2 \right\}.$$
He developed a method based on wavelet thresholding to estimate the densities of $X_i$ and $(X_i,X_{i+1})$ and showed that the quotient estimator of $s$ is quasi-optimal in the sense that the minimax rates are achieved up to possible logarithmic factors. 
\cite{Lacour2008} used model selection via penalization to construct estimates of the densities. The resulting quotient estimator reaches the minimax rates over $\F$ when  $\F_1$ and $\F_2$ are balls of homogeneous (but possibly anisotropic)  Besov spaces on $[0,1]^2$ and $[0,1]$ respectively.

The previous rates of convergence depend on the smoothness properties of the densities of $X_i$ and $(X_i,X_{i+1})$. 
In the favourable case where $X_0,\dots,X_n$ are drawn from a stationary Markov chain  (with stationary density $f$), the rates depend on the smoothness properties of $f$ or more precisely on the restriction of $f$ to $A_1$.  This function may  however be less regular than the target function $s$. We refer for instance to Section 5.4.1 of~\cite{Clemenccon2000} for an example of a Doeblin recurrent Markov chain where the stationary density $f$ is discontinuous on $[0,1]$ although~$s$ is constant on $[0,1]^2$. Therefore, these estimators may converge slowly  even if $s$ is smooth, which is problematic. 

This issue was overcome in several papers. \cite{Clemenccon2000} proposed a second procedure, based on wavelets and an analogy with the regression setting. He computed the lower bounds of minimax rates when the restriction of $s$ on $[0,1]^2$ belongs to  balls of some (possibly inhomogenous) Besov spaces and proved that its estimator achieves these rates up to a possible logarithmic factor.
\cite{Lacour2007} established lower bound over balls of some (homogenous but possibly anisotropic) Besov spaces.
By minimizing a penalized contrast inspired from the least-squares, she obtained a model selection theorem from which she deduced that her estimator reaches the minimax rates when $A = [0,1]^2$, $\XX^2 = \R^2$.
With a similar procedure, \cite{Akakpo2011} obtained the usual rates of convergence over balls of possibly anisotropic and inhomogeneous Besov spaces (when $\XX^2 = A = [0,1]^{2 d}$). Very recently,~\cite{Birge2012} proposed a procedure based on robust testing to establish a general oracle inequality. The expected rates of convergence can be deduced from this inequality when~$\sqrt{s}$ belongs to balls of possibly anisotropic and inhomogeneous Besov spaces.

These authors have used different losses in order to evaluate the performance of their estimators. In each of these papers, the risk of an estimator $\hat{s}$ is of the form $\E\left[\delta^2(s \1_A, \hat{s})\right]$  where~$\1_A$ denotes the indicator function of the subset $A$ and  $\delta$  a suitable distance.
\cite{Lacour2007}, \cite{Akakpo2011} considered the space $\L^2 (\XX^2, M)$ of square integrable functions on $\XX^2$ equipped with the random product measure $M =  \lambda_n \otimes \mu$ where  $\lambda_n =  n^{-1} \sum_{i=0}^{n-1} \delta_{X_i}$ and used the distance defined for $f,f' \in \L^2 (\XX^2, M)$  by 
\begin{eqnarray*}
\delta^2(f,f')= \frac{1}{n} \sum_{i=0}^{n-1} \int_{\XX} \left( f (X_i,y) - f'(X_i,y)  \right)^2 \d \mu (y).
\end{eqnarray*}
\cite{Birge2012} considered the cone $\L^1_+ (\XX^2, \mu \otimes \mu)$ of non-negative integrable functions and used the deterministic Hellinger-type distance defined for $f,f' \in  \L^1_+ (\XX^2, \mu \otimes \mu)$ by
 \begin{eqnarray*}
\delta^2(f,f')= \frac{1}{2} \int_{\XX^2} \left( \sqrt{f (x,y)} - \sqrt{f'(x,y)}  \right)^2 \d \mu (x) \d \mu (y).
\end{eqnarray*}
These approaches, which often rely on the loss that is used, require the knowledge (or at least a suitable estimation) of various quantities depending on the unknown $s$, such as the supremum norm of $s$, or on a positive lower bound, either on the stationary density, or on $k^{-1} \sum_{j=1}^k s^{(l + j)}$ for some $k \geq 1$, $l \geq 0$ where $s^{(l+j)} (x,\cdot)$ is the density of the conditional law $\mathcal{L} (X_{l+j} \mid X_0 = x)$.
Unfortunately, these quantities not only  influence the way the estimators are built but also their performances since they are involved in the risk bounds. In the present paper, we shall rather consider the distance $H$ (corresponding to an analogue of the random $\L^2$ loss above) defined on the cone $\L_{+}^1 (\XX^2, M)$ of integrable and non-negative functions  by 
\begin{eqnarray*}
H^2(f,f') = \frac{1}{2 n} \sum_{i=0}^{n-1} \int_{\XX} \left( \sqrt{f (X_i,y)} - \sqrt{f'(X_i,y)}  \right)^2 \d \mu (y) \quad \text{for all  $f,f' \in \L_{+}^1 (\XX^2, M)$.}
\end{eqnarray*}
For such a loss, we shall show that our estimators satisfy an oracle-type inequality under very weak assumptions on the Markov chain. A connection with the usual deterministic Hellinger-type loss will be done under a posteriori assumptions on the chain, and hence, independently of the construction of the estimator.

Our estimation strategy can be viewed as a mix between an approach based on the minimization of a contrast and an approach based on robust tests.  Estimation procedures based on tests  started in the seventies with Lucien Lecam and Lucien Birgé (\cite{LeCam1973,LeCam1975,Birge1983, Birge1984, Birge1984a}). More recently, \cite{BirgeTEstimateurs} presented a powerful device to establish general oracle inequalities from robust tests. It was used in our statistical setting in \cite{Birge2012}  and in many others in  \cite{BirgePoisson, BirgeDens}  and \cite{Sart2012}. 
We make two contributions to this area. Firstly, we provide a new test for our statistical setting. This test is based on a variational formula inspired from~\cite{BaraudMesure} and differs from the one of~\cite{Birge2012}. Secondly, we  shall study procedures that are quite far from the original one of~\cite{BirgeTEstimateurs}. Let us explain why.

The procedure of~\cite{BirgeTEstimateurs} depends on a suitable net, the construction of which is usually abstract, making thus the estimator impossible to build in practice. In the favourable cases where the net can be made explicit, the procedure is anyway too complex to be implemented (see for instance Section~3.4.2 of~\cite{BirgePoisson}). This procedure was afterwards adapted to estimators selection in~\cite{BaraudBirgeHistogramme} (for histogram type estimators) and in~\cite{BaraudMesure} (for  more general estimators). The complexity of their algorithms is of order the square of the cardinality of the family and are thus implementable when this family is not too large. In particular, given a family of histogram type estimators $\{\hat{s}_m,  m \in \M\}$, these two procedures are interesting in practice when $\M$ is a collection of regular partitions (namely when all its elements have same Lebesgue measure) but become unfortunately numerically intractable for richer collections.  In this work, we tackle this issue  by proposing  a new way of selecting among a  family of piecewise constant estimators when the collection $\M$ ensues from the adaptive approximation algorithm of~\cite{DeVore1990}.  

We present this procedure in the first part of the paper. It yields to a piecewise constant estimator on a data-driven   partition that satisfies an oracle-type inequality from which we shall deduce uniform rates of convergence over balls of (possibly) inhomogeneous Besov spaces with small regularity indices. These rates coincide, up to a possible logarithmic factor to the usual ones over such classes. Finally, we carry out numerical simulations to compare our estimator with the one of~\cite{Akakpo2011}. 

In the second part of this paper, we are interested in obtaining stronger theoretical results for our statistical problem. We put aside the practical considerations to focus on the construction of an estimator that satisfies a general model selection theorem. Such an estimator should be  considered as a benchmark for what theoretically feasible. We deduce rates of convergence over a large range of anisotropic and inhomogeneous Besov spaces on $[0,1]^{2 d}$.  We shall also consider other kinds of assumptions on the transition density. We shall assume that~$s$ belongs to classes of functions satisfying structural assumptions  and for which faster rates of convergence can be achieved. This approach was developed by~\cite{Juditsky2009} (in the Gaussian white noise model) and  by~\cite{BaraudComposite} (in more statistical settings)  to avoid the curse of dimensionality. More precisely,~\cite{BaraudComposite} showed that these rates can be deduced from a general model selection theorem, which strengthen its theoretical interest. This strategy was used in~\cite{Sart2012} to establish risk bounds over many classes of functions for Poisson processes with covariates. We shall use these assumptions to obtain faster rates of convergence for autoregressive  Markov chains (whose conditional variance may not  be constant).

This paper is organized as follows. The first procedure, which  selects among piecewise constant estimators is presented and theoretically  studied  in Section~2. In Section~3, we  carry out a simulation study and compare our estimator with the one of~\cite{Akakpo2011}. The practical implementation of this procedure is quite technical and will  therefore be delayed  in the appendix, in Section~\ref{SectionAlgorithm}.  In Section~4, we establish theoretical results by using our second procedure. The proofs are postponed to Section~6.

Let us introduce some notations that will be used all along the paper.  
The number $x \vee y$ (respectively $x \wedge y$) stands for $\max(x,y)$ (respectively $\min(x,y)$) and $x_+$ stands for $x \vee 0$.  We set $\N^{\star} = \N \setminus \{0\}$. For $(E,d)$ a metric space, $x \in E$ and $A \subset E$, the distance between $x$ and~$A$ is denoted by $d(x,A)= \inf_{a \in A} d(x,a)$. The indicator function of a subset $A$ is denoted by~$\1_A$ and the restriction of a function $f$ to $A$ by  $\restriction{f}{A}$. For all real valued function $f$ on $E$, $\|f\|_{\infty}$ stands for $\sup_{x \in E} |f(x)|$. The cardinality of a finite set $A$ is denoted by $|A|$. 
The notations $C$,$C'$,$C''$\dots are for the constants. The constants $C$,$C'$,$C''$\dots   may change from line to line.

\section{Selecting among piecewise constant estimators.} \label{SectionExPartition}
Throughout this section,  we assume that $\XX = \R^{d}$, $A  = [0,1]^{2 d}$, $\mu ([0,1]^{d}) = 1$ and $n > 3$. 

\subsection{Preliminary estimators.} Given a (finite) partition $m$ of $[0,1]^{2 d}$, a simple way to estimate~$s$ on $[0,1]^{2 d}$ is to consider the piecewise constant estimator on the elements of $m$ defined by
\begin{eqnarray} \label{EstimateurConstantParMorceaux}
\hat{s}_{m} = \sum_{K \in m} \frac{\sum_{i=0}^{n-1} \1_K (X_i,X_{i+1})}{\sum_{i=0}^{n-1} \int_{[0,1]^d} \1_K (X_i,x) \d \mu (x)} \1_K.
\end{eqnarray}
In the above definition, the denominator $\sum_{i=0}^{n-1} \int_{\XX} \1_K (X_i,x) \d \mu (x)$ may be equal to~$0$ for some sets~$K$, in which case the numerator $\sum_{i=0}^{n-1} \1_K (X_i,X_{i+1}) = 0$ as well, and we shall use the convention $0/0 = 0$. 

We now bound from above the risk of this estimator. We set $$V_m = \left\{ \sum_{K \in m} a_K \1_K ,\, \forall K \in m, \, a_K \in [0,+\infty) \right\}$$ 
and  prove the following.
\begin{prop}\label{prophisto}
For all finite partition $m$ of $[0,1]^{2 d}$,
\begin{eqnarray*}\label{eqhisto}
C \E \left[H^2(s \1_{A}, \hat{s}_{m})\right] \leq   \E \left[H^2 (s \1_{A}, V_{m})\right] + \frac{1 + \log n}{n} |m|
\end{eqnarray*}
where $C = 1/(4 + \log 2)$.
\end{prop}
Up to a constant, the risk of $\hat{s}_{m}$ is bounded by a sum of two terms. The first one corresponds to the approximation term whereas the second one corresponds to the estimation term.

An analogue upper bound on the empirical quadratic risk of this estimator may be found in Chapter 4 of~\cite{Akakpo2009}. Her bound requires several assumptions on the partition $m$ and the Markov chain although the present one requires none. However, unlike hers, we lose a logarithmic term.

\subsection{Definition of the partitions.} \label{SectionChoixM}
In this section, we shall deal with special choice of partitions $m$. More precisely,  we consider the family of partitions defined  by using the recursive algorithm developed in~\cite{DeVore1990}. 
For $j \in \N$, we consider the set 
$$\mathcal{L}_j = \left\{ \mathbf{l} = (l_1,\dots,l_{2 d}) \in \N^{2 d} , \, 1 \leq l_i \leq 2^j \; \text{for} \; 1 \leq i \leq 2 d\right\}$$
and define for all $\mathbf{l} = (l_1,\dots,l_{2 d})  \in \mathcal{L}_j$,
\begin{equation*}
 \forall i \in \{1,\dots,2 d\}, \quad I_j (l_i)  =
  \begin{cases}
  \left[\frac{l_i-1}{2^{j}}, \frac{l_i}{2^{j}} \right) & \text{if $l_i < 2^{j}$} \\
    \left[\frac{l_i-1}{2^{j}}, 1 \right]  & \text{if $l_i = 2^{j}$.} 
  \end{cases}
\end{equation*}
We then introduce the cube  $K_{j,\mathbf{l}} = \prod_{i=1}^{2 d}  I_j (l_i)$ and  set  $\mathcal{K}_j = \left\{ K_{j, \mathbf{l}}, \, \mathbf{l} \in \mathcal{L}_j \right\}$.

The algorithm starts with $[0,1]^{2 d}$. At each step, it gets a partition of $[0,1]^{2 d}$ into a finite family of disjoint cubes of the form $K_{j,\mathbf{l}}$. For any such cube, one decides  to divide it into the~$4^{d}$ elements of $\mathcal{K}_{j+1}$ which are contained in it, or not. 
The set of all such partitions that can be constructed in less than $\ell$ steps is denoted by $\M_{\ell}$. 
We set $\M_{\infty} = \cup_{\ell \geq 1} \M_{\ell}$.
Two examples of partitions are illustrated in  Figure~\ref{FigurePartitions} (for $d = 1$).

\begin{center}
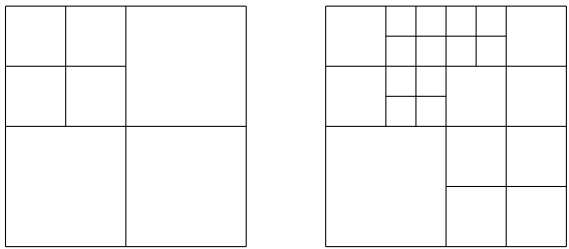
\begin{figure} [H] 
\setlength{\unitlength}{0.4cm}
\begin{picture}(10,8) 
\put(0,8){\line(1,0){8}}
\put(8,0){\line(0,1){8}}
\put(0,0){\line(0,1){8}}
\put(0,0){\line(1,0){8}}
\put(0,4){\line(1,0){8}}
\put(4,0){\line(0,1){8}}
\put(0,6){\line(1,0){4}}
\put(2,4){\line(0,1){4}}
\end{picture}
\setlength{\unitlength}{0.4cm}
\begin{picture}(10,8)
\put(0,8){\line(1,0){8}}
\put(8,0){\line(0,1){8}}
\put(0,0){\line(0,1){8}}
\put(0,0){\line(1,0){8}}
\put(0,4){\line(1,0){8}}
\put(4,0){\line(0,1){8}}
\put(0,6){\line(1,0){4}}
\put(2,4){\line(0,1){4}}
\put(2,5){\line(1,0){2}}
\put(3,4){\line(0,1){2}}
\put(2,7){\line(1,0){2}}
\put(3,6){\line(0,1){2}}
\put(6,4){\line(0,1){4}}
\put(4,6){\line(1,0){4}}
\put(4,7){\line(1,0){2}}
\put(5,6){\line(0,1){2}}
\put(4,2){\line(1,0){4}}
\put(6,0){\line(0,1){4}}
\end{picture}
\caption{Left: example of a partition of  $\M_2$. Right: example of a partition of  $\M_3$.} \label{FigurePartitions}
\end{figure}
\end{center}

\subsection{The  selection rule.} \label{SectionProcedure}
Given   $\ell \in \N^{\star} \cup \{\infty\}$, the aim of this section is to select an estimator among the family $\{\hat{s}_m, \, m \in \M_{\ell}\}$. 

For any  $K \in \cup_{m \in \M_{\ell}} m$ and any partition $m' \in \M_{\ell}$, let $m' \vee K $ be the partition of $K$ defined by 
$$m' \vee K = \{K' \cap K, \, K' \in m',\, K \cap K' \neq \emptyset \}.$$
Let $L$ be a positive number and $\pen$ be the non-negative map defined by
$$\pen \left(m' \vee K \right) = L \frac{|m' \vee K| \log n}{n} \quad \text{for all $m' \in \M_{\ell}$ and  $K \in \cup_{m \in \M_{\ell}} m$}.$$
Let us set $\alpha = (1-1/\sqrt{2})/2$ and for all $f,f' \in \L_+^1 (\XX^2, M)$,
\begin{eqnarray} \label{TestBaraud}
 T(f , f')\!\! &=& \!\!\! \frac{1}{2 n \sqrt{2} } \sum_{i=0}^{n-1}   \int_{\XX} \sqrt{f(X_i,y) + f'(X_i,y)} \left(\sqrt{f' (X_i,y)} - \sqrt{f(X_i,y)} \right) \d \mu (y)        \\
& & \!\!\!  +  \frac{1}{n \sqrt{2}}  \sum_{i=0}^{n-1}   \frac{\sqrt{f'(X_i,X_{i+1})} - \sqrt{f(X_i,X_{i+1})}}{\sqrt{f (X_i,X_{i+1}) + f'(X_i,X_{i+1})}} \nonumber \\
& & \!\!\! + \frac{1}{2 n} \sum_{i=0}^{n-1}   \int_{\XX} \left( f(X_i,y) - f'(X_i,y) \right) \d \mu (y)  \nonumber.
\end{eqnarray}
We define $\gamma$  for $m \in \M_{\ell}$ by
\begin{eqnarray*}
\gamma(m) \!\!\!   &=& \!\!\!  \left\{ \sum_{K \in m} \sup_{m' \in \M_{\ell}} \left[\left(  \sum_{K' \in m'} \left( \alpha H^2 \left(\hat{s}_m \1_{K \cap K'}, \hat{s}_{m'}\1_{K \cap K'}\right) +  T \left(\hat{s}_m\1_{K \cap K'},\hat{s}_{m'}\1_{K \cap K'} \right)\right) \right)  \right. \right. \\
& & \left. \left. \qquad -   \pen (m' \vee K)   \right] \right\}  + 2 \pen(m).
\end{eqnarray*}
Finally, we select $\hat{m}$ among $\M_{\ell}$ as any partition satisfying
\begin{eqnarray} \label{eqAminimiser}
\gamma(\hat{m}) \leq \inf_{m \in \M_{\ell}} \gamma(m) + \frac{1}{n}
\end{eqnarray}
and consider the resulting estimator $\hat{s}  = \hat{s}_{\hat{m}}$.

\paragraph{Remarks.}
The estimator $\hat{s} = \hat{s} (L,\ell)$ depends on the choices of two quantities $L > 0$, $\ell \in \N^{\star} \cup \{\infty\}$. We shall see in the next section that $L$ can be chosen as an universal numerical constant.
As to~$\ell$, from a theoretical point of view, it can be chosen as $\ell = \infty$. In practice, we recommend to take it as large as possible. Nevertheless, the larger $\ell$, the longer it takes to compute the estimator.
A practical algorithm in view of computing $\hat{m}$ will be detailed in the appendix.

The selection procedure we use may look somewhat unusual. It can be seen as a mix between a procedure based on a contrast function (which is usually easy to implement) and a procedure based on a robust test (the functional $T(f,f')$, which can be seen as a robust test between $f,f'$, will allow us to obtain risk bounds with respect to a Hellinger-type distance).
This functional is  inspired from the variational formula for the Hellinger affinity described in Section~2 of~\cite{BaraudMesure}.

\subsection{An oracle inequality.} \label{SectionOracleConstant}
The main result of this section is the following.
\begin{thm}  \label{CorTheoremSelectionHisto} 
There exists an universal constant $L_0 > 0$ such that, for all $L \geq L_0$,  $\ell \in \N^{\star} \cup \{\infty\}$,
 the estimator  $\hat{s} = \hat{s} (L, \ell)$   satisfies 
\begin{eqnarray} \label{EqOracleSelectHisto}
C \E \left[H^2 \left(s \1_{A}, \hat{s} \right) \right] \leq  \inf_{m \in \M_{\ell}} \left\{  \E \left[H^2 \left(s \1_{A}, V_m \right)\right] + \pen(m)  \right\} 
\end{eqnarray}
where $C$ is an universal positive constant. 
\end{thm}
In the literature, oracle inequalities with a random quadratic loss for piecewise constant estimators have been obtained in~\cite{Lacour2007} and~\cite{Akakpo2011}. Their procedures require a priori assumptions on the transition density and  the Markov chain although ours requires none (except homogeneity). However, unlike theirs, our risk bound involves an extra  logarithmic term. We do not know whether this term is necessary or not.

In the proof, we obtain an upper bound for $L_0$ which is unfortunately very rough and useless in practice. It seems difficult to obtain a sharp bound on $L_0$ from the theory and we have rather  carried out a simulation study in order to tune $L_0$ (see Section~\ref{SectionSimus}).

\subsection{Risk bounds with respect to a deterministic loss.} \label{SectionHellingerDeterministe} Although the distance $H$ is natural, we are interested in  controlling  the risk  associated to a deterministic distance. 
To do so, we shall make a posteriori assumptions on the Markov chain. 

\begin{hyp} \label{Hypothese1}
The sequence $(X_i)_{i \geq 0}$ is stationary and admits a stationary density~$\varphi$ with respect to the Lebesgue measure $\mu$ on $\R^d$.
There exists $\kappa_0 > 0$ such that $\varphi(x) \geq \kappa_0$ for all $x \in [0,1]^d$.
\end{hyp}
We introduce $\LL_+^1 \left([0,1]^{2 d}, (\varphi \cdot \mu) \otimes \mu\right)$ the cone of integrable and non-negative functions on $[0,1]^{2 d}$ with respect to the product measure $(\varphi \cdot \mu) \otimes \mu$. We endow $\LL_+^1 ([0,1]^{2 d}, (\varphi \cdot \mu) \otimes \mu)$  with the distance~$h$ defined by
$$\forall f,f' \in \LL_+^1 \big([0,1]^{2 d}, (\varphi \cdot \mu) \otimes \mu\big), \quad  h^2 (f,f') =  \frac{1}{2} \int_{[0,1]^{2 d }} \left(\sqrt{f (x,y)} - \sqrt{f'(x,y)} \right)^2 \varphi (x) \d x \d y.$$
In our results, we shall need the $\beta$-mixing properties of the Markov chain. We  set for all $q \in \N^{\star}$
$$\beta_q = \int_{\R^{2 d}}  |s^{(q)} (x,y) - \varphi (y)   |   \varphi (x) \d x \d y$$ 
where $s^{(q)}(x,\cdot)$ is the  density of the conditional law $\mathcal{L} (X_q \mid X_0 = x)$ with respect to the Lebesgue measure.  
We refer to~\cite{Doukhan1994} and~\cite{Bradley2005} for more details on the $\beta$-mixing coefficients.

\begin{thm}  \label{CorTheoremSelectionHistoDistanceDeterministe}
Under Assumption~\ref{Hypothese1}, the estimator $\hat{s}$ built in Section~\ref{SectionProcedure} with $\ell \in \N^{\star}$ and $L \geq L_0$, satisfies 
\begin{eqnarray*} \label{InequaliteOracleDet}
C \E \left[h^2 \left(s \1_{A}, \hat{s} \right) \right] \leq  \inf_{m \in \M_{\ell}} \left\{  h^2 \left(s \1_{A}, V_m \right) + \pen(m) \right\} + \frac{R_n (\ell)}{n}
\end{eqnarray*}
where 
\begin{eqnarray} \label{DefRnBetaMixing}
R_n (\ell) = n 2^{3 \ell d}  \inf_{1 \leq q \leq n} \left\{ \exp \left(- \frac{\kappa_0}{10} \frac{n}{q 2^{\ell d}} \right) + n  \beta_q / q \right\}
\end{eqnarray}
and where $C$ is an universal positive constant.
\end{thm}
This result is interesting when the remainder term $R_n  (\ell)/ n$ is small enough, that is  when  $2^{\ell d}$ is small compared to $n$ and when the sequence $(\beta_q)_{q \geq 1}$ goes to  $0$ fast enough. More precisely, $R_n (\ell)$ can be bounded independently of $n$, $\ell$ whenever  $\ell$, $d$, $n$ and the $\beta_q$ coefficients satisfy the following.

\begin{itemize}
\item If the chain is geometrically $\beta$-mixing, that is if there exists $b_1 > 0$ such that $\beta_q \leq  e^{-b_1 q}$, then 
$$R_n (\ell) \leq  n^2 2^{3 \ell d+1} \left[\exp (- b_1 n) + \exp \left(- \frac{\kappa_0}{10} \frac{n}{2^{\ell d}}  \right) + \exp \left(- \sqrt{\frac{\kappa_0 b_1}{40} \frac{n}{2^{\ell d}}}  \right)  \right].$$
In particular, if $\ell, d, n$ are such that $2^{\ell d} \leq n / \log^{3} n$, $R_n (\ell)$ is upper bounded by a constant depending only on $\kappa_0,b_1$.
\item  If the chain is arithmetically $\beta$-mixing, that is if there exists $b_2 > 0$ such that $\beta_q \leq  q^{-b_2}$, then 
$$R_n (\ell) \leq  \frac{C' (b_2)}{\kappa_0^{b_2+1}}  \frac{2^{(4 + b_2) \ell d} \log^{b_2+1} \left(1 + \frac{\kappa_0 n }{2^{\ell d}} \right)}{ n^{b_2 - 1}}$$
where $C' (b_2)$ depends only on $b_2$.
Consequently, if  $2^{\ell d} \leq n^{1 - \zeta} / \log n$  and  $b_2 \geq 5/\zeta-4$ for  $\zeta \in (0,1)$,  $R_n (\ell)$ is upper bounded by a constant depending only on $\kappa_0,b_2$.
\end{itemize}

\subsection{Rates of convergence.} \label{SectionRatesConvergence}
The aim of this section is to obtain uniform risk bounds over classes of smooth transition densities for our estimator.

\subsubsection{Hölder spaces.}
Given $\sigma \in (0,1]$, we say that a function $f$ belongs to the Hölder space $\mathcal{H}^{\sigma} ([0,1]^{2 d})$ if there exists $|f|_{\sigma} \in \R_+$ such that for all $(x_1,\dots,x_{2 d}) \in [0,1]^{2 d}$ and all $1 \leq j \leq 2 d$, the functions $f_j (\cdot)  = f (x_1,\dots,x_{j-1},\cdot,x_{j+1},\dots,x_{2 d})$ satisfy
$$ \left| f_j (x) - f_j (y) \right| \leq |f|_{\sigma} |x - y|^{\sigma} \quad  \text{for all $x,y \in [0,1]$.}$$
When the restriction of $\sqrt{s}$ to $A = [0,1]^{2 d}$ is Hölderian, we deduce from~{(\ref{EqOracleSelectHisto})} the following.

\begin{cor}
For all $\sigma \in (0,1]$ and $\restriction{\sqrt{s}}{A} \in \mathcal{H}^{\sigma} ([0,1]^ {2 d})$, the estimator $\hat{s} = \hat{s} (L_0,\infty)$ satisfies
\begin{eqnarray*} \label{EqVitesseHisto1}
C \E \left[H^2 \left(s \1_A, \hat{s} \right) \right] \leq  \left(d \left|\restriction{\sqrt{s}}{A}\right|_{\sigma}\right)^{\frac{2 d}{d +  \sigma}} \left(\frac{\log n }{n} \right)^{\frac{ \sigma}{\sigma + d}}   + \frac{\log  n}{n} 
\end{eqnarray*}
where $C$ is an universal positive constant.
\end{cor}

\subsubsection{Besov spaces.}
A thinner way to measure the smoothness of the transition density is to assume that $\restriction{\sqrt{s}}{A}$ belongs to a Besov space. We refer to Section~3 of~\cite{DeVore1990} for a definition of this space.
We say that the Besov space $\mathscr{B}^{\sigma}_q (\L^p ([0,1]^{2 d}))$  is homogeneous when $p \geq 2$ and inhomogeneous otherwise.  
We set for all $p \in (1,+\infty)$ and $\sigma \in (0,1)$,
\begin{equation*}
 \mathscr{B}^{\sigma} (\L^p ([0,1]^{2 d}))  =
  \begin{cases}
     \mathscr{B}^{\sigma}_{p} (\L^p ([0,1]^{2 d}))  & \text{if $p \in (1,2)$}  \\
  \mathscr{B}^{\sigma}_{\infty} (\L^p ([0,1]^{2 d})) & \text{if $p \in [2,+\infty)$,} 
  \end{cases}
\end{equation*}
and  denote by $|\cdot|_{p,\sigma}$ the semi norm of $\mathscr{B}^{\sigma} (\L^p ([0,1]^{2 d}))$. 
We make the following assumption to deduce  from~{(\ref{EqOracleSelectHisto})} risk bounds over  these  spaces.
\begin{hyp} \label{hypDensiteDesXi1}
There exists $\kappa > 0$ such that for all $i \in \{0,\dots,n-1\}$, $X_i$  admits a  density~$\varphi_i$  with respect to the Lebesgue measure $\mu$ such that $\varphi_i (x)\leq \kappa$ for all $x \in [0,1]^{d}$.
\end{hyp}
Note that we do not require that the chain be either stationary or mixing. 

Let  $\left(\LL^2 ([0,1]^{2 d}, \mu \otimes \mu),d_2\right)$, be the metric space of square integrable functions on~$[0,1]^{2 d}$ with respect to the Lebesgue measure. Under the above assumption, we deduce from~(\ref{EqOracleSelectHisto}) that
\begin{eqnarray*}
C \E \left[H^2 \left(s\1_A, \hat{s}_{\hat{m}} \right) \right] \leq  \inf_{m \in \M_{\ell}} \left\{ \kappa d^2_2 \left(\restriction{\sqrt{s}}{A} , V_m \right)  + L_0 \frac{|m| \log  n }{n}  \right\}.
\end{eqnarray*}
When  $\restriction{\sqrt{s}}{A}$ belongs to a Besov space, the right-hand side of this inequality can be upper bounded thanks to the approximation theorems of~\cite{DeVore1990}.

\begin{cor}  \label{CorVitesseBesovIsotropes1}
Suppose that Assumption~\ref{hypDensiteDesXi1} holds.
For all $p \in (2 d / (d+1), +\infty)$, $\sigma \in (2 d (1/p - 1/2)_+, 1)$ and $\restriction{\sqrt{s}}{A} \in \mathscr{B}^{\sigma} (\L^p ([0,1]^ {2 d}))$, the estimator $\hat{s} = \hat{s} (L_0,\infty)$ satisfies
\begin{eqnarray} \label{EqVitesseHisto1}
C' \E \left[H^2 \left(s \1_A, \hat{s} \right) \right] \leq  \left|\restriction{\sqrt{s}}{A}\right|_{p,\sigma}^{\frac{2 d}{d +  \sigma}} \left(\frac{\log n }{n} \right)^{\frac{ \sigma}{\sigma + d}}   +  \frac{\log  n}{n} 
\end{eqnarray}
where $C' > 0$ depends only on $\kappa$,$\sigma$,$d$,$p$.
\end{cor}
More precisely, it is shown in the proof that the estimators $\hat{s} = \hat{s} (L_0,\ell)$ satisfy~{(\ref{EqVitesseHisto1})} when $\ell$ is large enough (when $\ell \geq d^{-1} (\log 2)^{-1} \log n$).

Rates of convergence for the deterministic loss $h$ can be established by using Theorem~\ref{CorTheoremSelectionHistoDistanceDeterministe} instead of Theorem~\ref{CorTheoremSelectionHisto}. 
For instance, if the chain is geometrically $\beta$-mixing, we may choose $\ell$ the smallest integer larger than $d^{-1} (\log 2)^{-1} \log (n / \log^3 n)$, in which case the estimator $\hat{s} = \hat{s} (L_0,\ell)$ achieves the  rate $(\log n / n)^{\sigma / (\sigma + d)}$ over the Besov spaces $\mathscr{B}^{\sigma} (\L^p ([0,1]^ {2 d}))$, $p \in (2 d / (d+1), +\infty)$, $\sigma \in (\sigma_1 (p,d), 1)$ where
\begin{eqnarray*}
\sigma_1 (p,d) = \frac{d}{4} \left(-1+4 \left(1/p - 1/2 \right)_+ +\sqrt{1+24 \left(1/p - 1/2 \right)_+ +16  \left(1/p - 1/2 \right)_+^2}\right).
\end{eqnarray*}
If the chain is arithmetically $\beta$-mixing with $b_q \leq q^{-6}$,  choosing $\ell$ the smallest integer larger than $d^{-1} (2 \log 2)^{-1} \log (n / \log n)$ allows us to recover the same rate of convergence when  $\sigma \in (\sigma_2 (p,d), 1)$
where
\begin{eqnarray*}
\sigma_2 (p,d) = d \left((1/p-1/2)_+ + \sqrt{2 (1/p-1/2)_+ + (1/p-1/2)_+^2} \right).
\end{eqnarray*}
We refer the reader to Section~\ref{SectionPreuvesCasHDeterministe} for a proof of these two results. 

In the literature,~\cite{Lacour2007} obtained a rate of order $n^{- \sigma / (\sigma + 1)}$ over $\mathscr{B}^{\sigma} (\L^2 ([0,1]^2))$, which is slightly faster but her approach prevents her to deal with inhomogeneous Besov spaces and requires the prior knowledge of a suitable upper bound on the supremum norm of $s$. 
As far as we know, the rates that have been established in the other papers hold only when $\sigma > 1$.

\section{Simulations.} \label{SectionSimus}
In this section, we present a simulation study to evaluate the performance of our estimator in practice. We shall simulate several Markov chains and estimate their transition densities by using our procedure.

\subsection{Examples of Markov chains.} \label{SectionExMarkovChains}
We consider Markov chains of the form
$$X_{k+1} = F(X_k, U_k)$$
where $F$ is some known function and where $U_k$ is a random variable  independent of $(X_0,\dots,X_{k})$.

For the sake of comparison, we begin to deal with examples that have already been considered in the simulation study of~\cite{Akakpo2011}.  In each of these examples, $U_k$ is a standard Gaussian random variable.

\begin{exemple}~\label{ExempleAr}
$X_{k+1} = 0.5 X_k +  (1+U_k)/4$
\end{exemple}

\begin{exemple}~\label{ExempleArch} 
$X_{k+1} = 12^{-1} \left(6 + \sin(12 X_k - 6) + (\cos(X_k - 6) + 3) U_k \right)$
\end{exemple}

\begin{exemple}~\label{ExempleBeta} 
$$X_{k+1} = \frac{1}{3} \left(X_k + 1 \right) + \left(\frac{1}{9} - \frac{1}{23} \left( \frac{1}{2} \beta (5 X_i/3,4,4) + \frac{1}{20} \beta \left( (5 X_i-2)/3,400, 400 \right) \right) \right) U_k$$
where  $\beta(\cdot,a,b)$ is the density of the $\beta$ distribution with parameters $a$ and $b$.
\end{exemple}

\begin{exemple}~\label{ExempleBeta2} 
$$X_{k+1} = \frac{1}{4} \left( g(X_k) + 1 \right) + \frac{1}{8} U_k$$
where $g$ is defined by
$$g(x) =  \frac{9 \sqrt{2}}{4 \sqrt{\pi}} \exp \left(-18 (x - 1/2)^2 \right) +   \frac{9 \sqrt{2}}{4 \sqrt{ \pi}} \exp \left(-162 (x - 3/4)^2 \right)  \quad \text{for all $x \in \R$.}$$
\end{exemple}
At first sight, Examples~1 and~2 may seem to be different than those of~\cite{Akakpo2011}.  Actually, we just have rescaled the data in order to estimate on $[0,1]^2$. The statistical problem is  the same.
According to~\cite{Akakpo2011}, we set $p$ large ($p = 10^4$) and we estimate the transition densities of Examples~1,~2,~3 and~4 from $(X_p,\dots,X_{n+p})$ so that the chain is approximatively stationary.

We also propose to consider the following examples. In Example~5, $U_k$ is a centred Gaussian random  variable with variance $1/2$, in Example~6, $U_k$  admits the density
$$f(x) = \frac{5 \sqrt{2}}{ 2 \sqrt{\pi} } \left[\exp \left(- 50 (x-1)^2 \right) +  \exp \left(- 50 x^2 \right) \right]$$ 
with respect to the Lebesgue measure, and in Example~7, $U_k$ is an  exponential random variable with parameter $1$. 
\begin{exemple}~\label{ExempleAr3}
$X_{k+1} = 0.5 X_k + (1+U_k)/4.$
\end{exemple}
\begin{exemple}~\label{ExempleAr2}
$X_{k+1} = 0.5 \left(X_k + U_k \right).$
\end{exemple}
\begin{exemple}~\label{Exemple5}
$X_{k+1} =X_k / (50 X_k + 1) + X_k U_k.$
\end{exemple}
We set $X_0 = 1/2$ and estimate $s$ from $(X_0,\dots,X_n)$.
These last three Markov chains are not stationary. Their transition densities are rather isotropic and inhomogeneous. The transition density of Example~7 is unbounded.

In what follows, our selection rule will always be applied with $L = 0.03$ (whatever, $\ell$, $n$ and the Markov chain). 

\subsection{Choice of $\ell$.}
We discuss the choice of $\ell$ by simulating the preceding examples with $n = 10^3$ and by applying our selection rule  for each value of $\ell \in \{1,\dots,10\}$. The results are summarized below.

\begin{figure}[H] 
\centering 
\begin{tabular}{|c|c|c|c|c|c|c|c|}
\hline 
$\ell$ & Ex 1 & Ex 2 & Ex 3 & Ex 4 & Ex 5 & Ex 6 & Ex 7 \\ 
\hline 
1 & 0.031 &  0.046  &     0.299 & 0.181 &   0.089 &   0.291  &0.358 \\ 
\hline 
2 & 0.011   &  0.015 &    0.087 &  0.107 &   0.024 & 0.170 & 0.241 \\ 
\hline 
3 & 0.011   &  0.014 &  0.026 &  0.058 & 0.013  &0.067 &  0.156 \\ 
\hline 
4 & 0.011  &  0.018 &  0.026 &0.035 & 0.015  & 0.046&  0.113 \\ 
\hline 
5 & 0.011  & 0.018 &  0.022 &   0.038 & 0.015 & 0.048 &  0.098 \\ 
\hline 
6 &  0.011 & 0.018  &  0.022 &  0.038 & 0.015  & 0.048  & 0.065 \\ 
\hline 
7 & 0.011 &  0.018  & 0.024 &   0.038 & 0.015 & 0.048 & 0.044 \\ 
\hline 
8 & 0.011 &  0.018  & 0.024 &   0.038 & 0.015 & 0.048 &  0.040\\ 
\hline 
9 & 0.011  &  0.018  &0.024  &   0.038 &0.015 & 0.048 &  0.040 \\ 
\hline 
10 & 0.011 & 0.018  &0.024  &   0.038& 0.015 &  0.048 & 0.040\\ 
\hline 
\end{tabular} 
\caption{Hellinger risk $H^2(s \1_{[0,1]^{2}},\hat{s})$.} \label{FigureRiskLVarie}
\end{figure}
When $\ell$ grows up, the risk of our estimator tends to decrease and then stabilize. The best choice of $\ell$ is obviously unknown in practice but this array shows that a good way for choosing~$\ell$ is to take it as large as possible.  This is theoretically justified by Theorem~\ref{CorTheoremSelectionHisto} since  the right-hand side of inequality~{(\ref{EqOracleSelectHisto})} is a non-increasing function of~$\ell$. 

\subsection{An illustration.}
We  apply our procedure for Examples~1 and~6  with $n = 10^4$, $\ell = 7$. We get two estimators and draw them with the corresponding transition density in Figure~\ref{FigureFunction}.

\begin{figure}[H] 
   \begin{minipage}[l]{0.5\linewidth}
     \centering \includegraphics[scale=0.19]{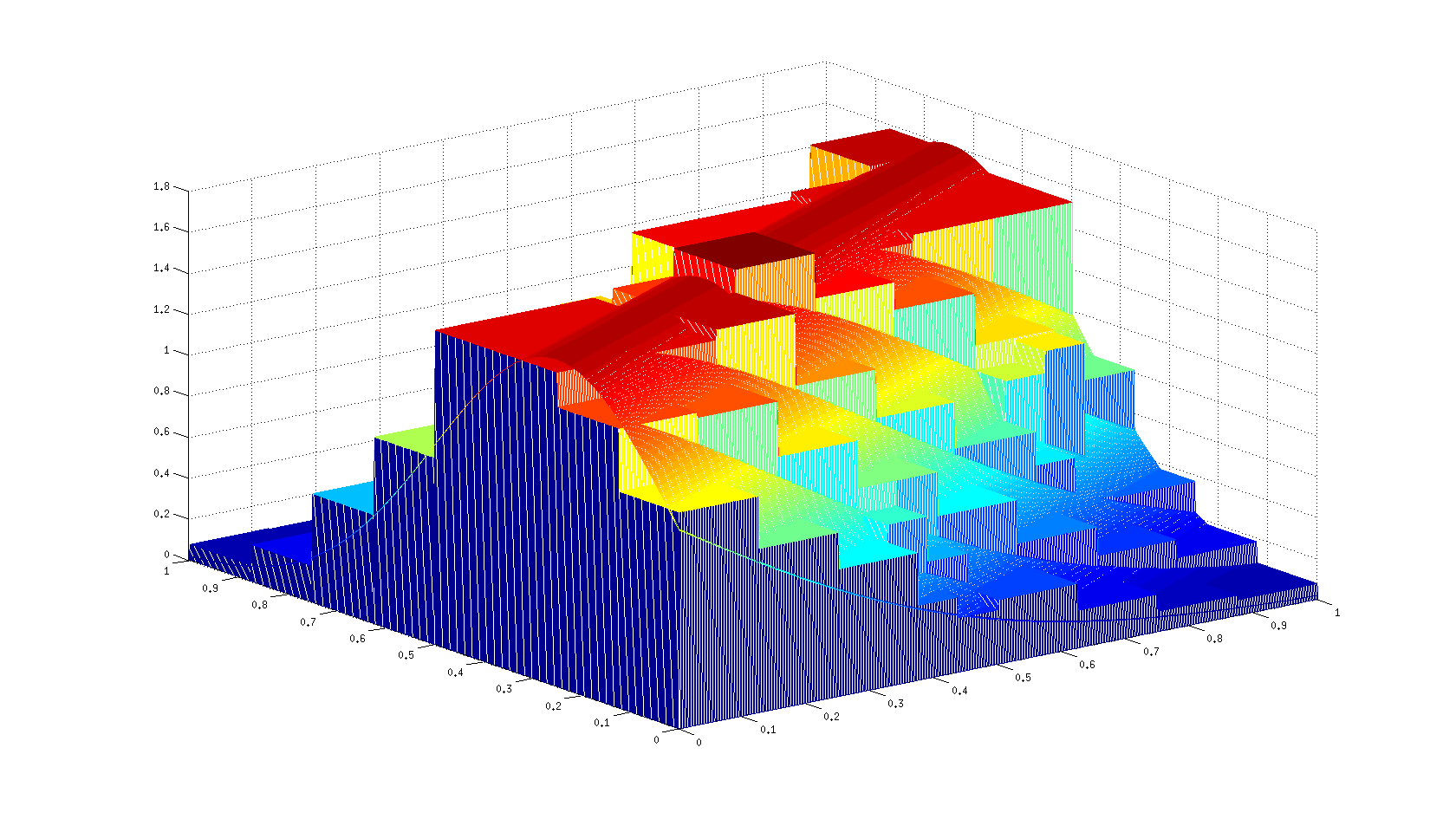}
		\begin{center}
     	 Example~1.
		\end{center}      

   \end{minipage}\hfill
   \begin{minipage}[l]{0.5\linewidth}   
      \centering \includegraphics[scale=0.19]{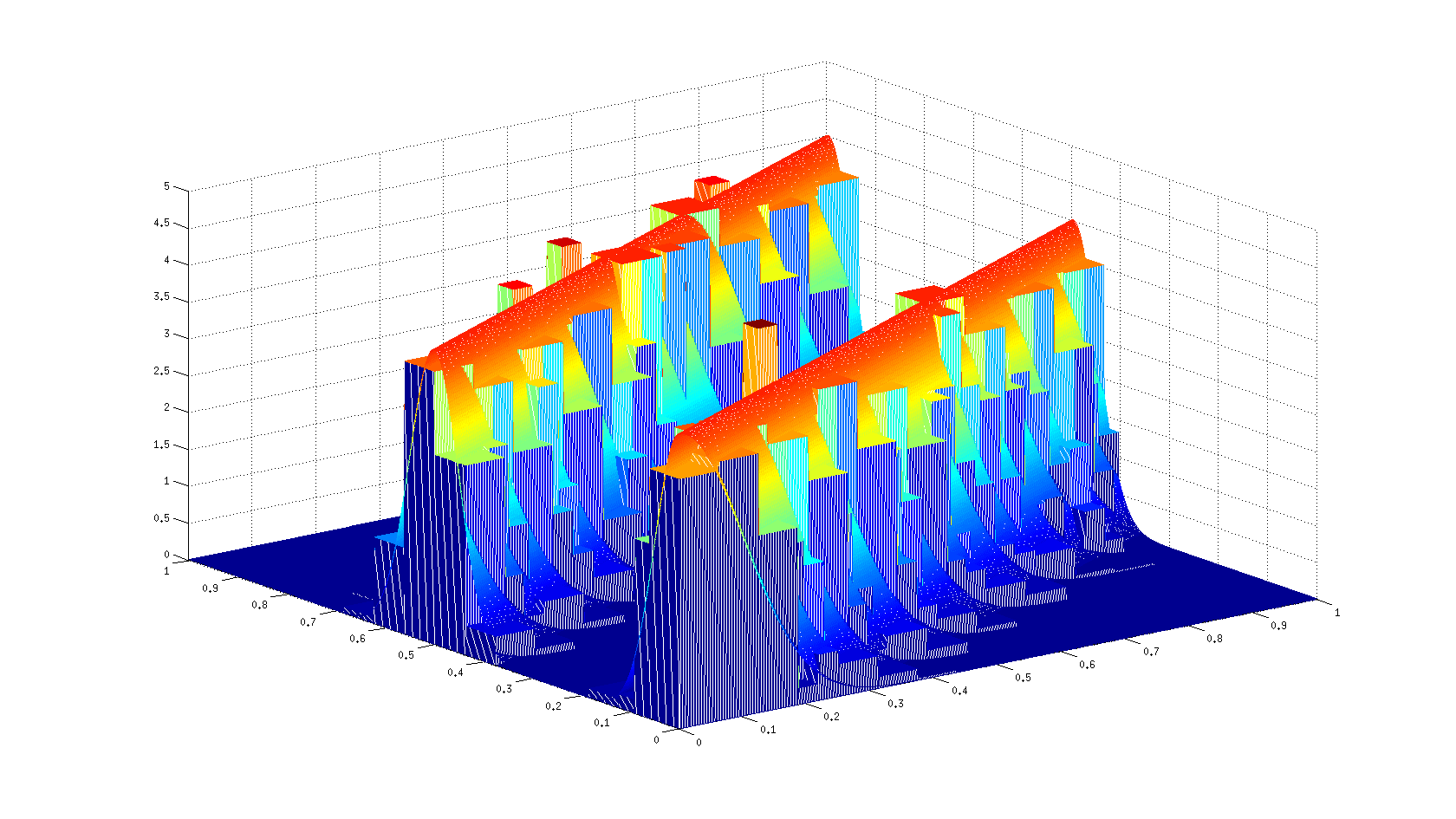}
      		\begin{center}
      Example~6.
		\end{center}      
   \end{minipage}
   \caption{Estimator and transition density.}  \label{FigureFunction}
\end{figure}
This shows that the selected partition is thinner (respectively wider) to the points where the transition density is changing rapidly (respectively slower), and is thus rather well adapted to the target function~$s$.

\subsection{Comparison with other procedures.}
In this section, we  compare our selection rule with the oracle estimator and with the piecewise constant estimator of~\cite{Akakpo2011}.

The procedure of~\cite{Akakpo2011} amounts to selecting an estimator among  $\{\hat{s}_m, m \in \M' \}$  where $\hat{s}_m$ is  defined by~{(\ref{EstimateurConstantParMorceaux})} and where $\M'$ is a collection of irregular partitions on $[0,1]^2$.
Precisely, with their notations, we apply it with $J_{\star}  = 5$, $\pen(m) = 3 \|\restriction{s}{A}\|_{\infty} |m|/n$ and with  $\pen(m) = 3 \|\hat{s}_{m^{\bullet}}\|_{\infty} |m|/n$ where $m^{\bullet}$ is a partition suitably chosen (following the recommendations of~\cite{Akakpo2011}, that is $J_{\bullet} = 3$). These two estimators are denoted by $\hat{s}^{(1)}$ and $\hat{s}^{(2)}$ respectively. Notice that these penalties, which are used in their simulation study, are not the ones prescribed by their theory. Their theoretical penalties also depend on a positive lower bound on the stationary density.
 
We denote by $\hat{s}^{(0)}$ the oracle estimator, that is the estimator defined as being a minimizer of the map $m \mapsto H^2(s 1_{[0,1]^2},\hat{s}_m)$ for $m \in \M_7$. This estimator is the best estimator of the family $\{\hat{s}_m, \, m \in \M_7\}$ and is known since the data are simulated.  
We consider the random variables
$$\mathcal{R}_i = \frac{H^2(s \1_{[0,1]^{2}},\hat{s})}{H^2(s \1_{[0,1]^{2}},\hat{s}^{(i)})} \quad \text{for $i = 0,1,2$}$$
and denote by $q_{0} (\alpha)$ the $\alpha$-quantile of  $\mathcal{R}_{0}$. Results obtained are given in Figure~\ref{figureriskcomparaison}. 

\begin{figure}[H]
\begin{center}
\begin{tabular}{|c|c|c|c|c|c|c|c|}
\hline 
  & Ex 1 & Ex 2  & Ex 3  & Ex 4  & Ex 5  & Ex 6  & Ex 7 \\ \hline 
 $\E [H^2(s \1_{[0,1]^2},\hat{s})]$  &   0.011 &    0.017 &   0.022 &  0.038  &  0.018 &  0.052  &    0.049 \\ \hlinewd{1.2pt} 
 $\E [H^2(s \1_{[0,1]^2},\hat{s}^{(0)})]$   & 0.007  & 0.011  & 0.015 &  0.028  & 0.012  &    0.037  &   0.041 \\ \hline   
  $q_{0} (0.5) $  &  1.473  & 1.513 & 1.443  &  1.369  & 1.422  &   1.420    &  1.200 \\ \hline   
 $q_{0} (0.75) $  &  1.698 &   1.627  &   1.557  & 1.440  &    1.575  &   1.481 &  1.244  \\ \hline   
 $q_{0} (0.9) $  & 1.921  &  1.834 & 1.683  & 1.509  &  1.749 &     1.543 &   1.290  \\ \hline   
 $q_{0} (0.95) $  &  2.113  &  1.965  & 1.770  & 1.558  &   1.839 &   1.590   &   1.317  \\ \hlinewd{1.2pt}
  $\E [H^2(s \1_{[0,1]^2},\hat{s}^{(1)})]$  &  0.017  & 0.018   &  0.028  &   0.058  &  0.024  &   0.103   &  - \\ \hline
 $\P \left(\mathcal{R}_{2} \leq 1\right) $ &   0.964  & 0.740   & 0.908  & 1  &  0.984 &  1  &  - \\ \hlinewd{1.2pt}
  $\E [H^2(s \1_{[0,1]^2},\hat{s}^{(2)})]$  &  0.013 &   0.018 &     0.028 &     0.062  &  0.023 &    0.096  &  0.133   \\ \hline
 $\P \left(\mathcal{R}_{3} \leq 1\right) $  &   0.832 &    0.748  & 0.928 & 1  &   0.948  &  1  &   1 \\ \hline  
 \end{tabular}
\end{center}
\caption{Risks for simulated data with $n = 1000$ averaged over $250$ samples.
}  \label{figureriskcomparaison}
\end{figure}

\subsection{Comparison with a quadratic empirical risk.}
In~\cite{Akakpo2011}, the risks of the estimators are evaluated with a empirical quadratic norm and we can also compare the performances of our estimator to theirs by using this risk. 

To do so, let us denote by $\|\cdot \|_n$ the empirical quadratic norm defined by
$$\|f\|_n^2 = \frac{1}{n} \sum_{i=1}^n \int_{\R} f^2(X_i,x) \d x \quad \text{for all $f \in \LL^2 (\R^2, M)$}$$
and set for $i \in \{1,2\}$,
$$\mathcal{R}^{'}_i = \frac{\|s \1_{[0,1]^2} - \hat{s} \|_n^2}{\|s \1_{[0,1]^2} - \hat{s}^{(i)}\|_n^2} .$$
The results obtained are presented in Figure~\ref{figureriskcomparaison2}. They are  very similar to those of Figure~\ref{figureriskcomparaison}.
\begin{figure}[H]
\begin{center}
\begin{tabular}{|c|c|c|c|c|c|c|c|}
\hline 
  & Ex 1 & Ex 2  & Ex 3  & Ex 4  & Ex 5  & Ex 6  & Ex 7 \\ \hline 
 $\E [\|s \1_{[0,1]^2} - \hat{s}\|_n^2] $ & 0.064 & 0.108  &    0.229&     0.319 & 0.116 &    0.528  &  2.82   \\ \hlinewd{1.2pt} 
  $\E [\|s \1_{[0,1]^2} - \hat{s}^{(1)}\|_n^2] $  &  0.147 &  0.133 &  0.257 &      0.423 &     0.205 &  0.743  &-   \\ \hline
 $\P (\mathcal{R}_{2}^{'} \leq 1) $  & 0.980  &  0.820 & 0.788 &   0.984  &    0.992 & 1   & -  \\ \hlinewd{1.2pt}  
  $\E [\|s \1_{[0,1]^2} - \hat{s}^{(2)}\|_n^2]  $ & 0.091 & 0.129  &  0.262 &     0.418 &  0.159 &     0.739  &  6.08  \\ \hline
 $\P (\mathcal{R}_{3}^{'} \leq 1) $    &  0.864 &  0.780 &  0.792 &  0.980  &    0.940 &  1  &  1 \\ \hline  
 \end{tabular}
\end{center}
\caption{Risks for simulated data with $n = 1000$ averaged over $250$ samples.
}  \label{figureriskcomparaison2}
\end{figure}

\section{A general procedure.} \label{SectionTheoremesGeneraux}
In Section~\ref{SectionExPartition}, we used our selection rule to establish the oracle inequality~{(\ref{EqOracleSelectHisto})}, from which we deduced rates of convergence over  Besov spaces~$\mathscr{B}^{\sigma} (\L^p ([0,1]^{2 d}))$ with~$\sigma$ lower than $1$.
We now aim at obtaining rates for more general spaces of functions.
This includes Besov spaces with regularity index larger than $1$ and spaces corresponding to structural assumptions on $s$. We propose a second procedure to reach this goal.

The Markov chain takes its values into $\XX$ and we estimate~$s$ on a subset $A$ of the form $A = A_1 \times A_2$. We always assume that $n > 3$.

\subsection{Procedure and preliminary result.} \label{SectionProcedurePreliminaire}
Our second procedure is defined as follows.
Let  $\alpha = (1-1/\sqrt{2})/2$, $L > 0$, $S$ be an at most countable set of $\L_+^1 (\XX^2, M)$ and $\Delta_S \geq 1$ be a map on $S$.

We define the application $\wp$ on $S$ by
$$\wp(f) = \sup_{f' \in S} \left[\alpha H^2 (f,f') + T(f,f') - L \frac{\Delta_S(f')}{n} \right] +  L \frac{\Delta_S(f)}{n} \quad \text{for all $f \in S$.}$$
We select $\hat{s}$ among $S$ as any element of $S$ satisfying
$$\wp(\hat{s}) \leq \inf_{f \in S} \wp(f) + \frac{1}{n}.$$
We prove the following.
\begin{prop} \label{ThmGeneral}
Suppose that  $f(x) = 0$ for all  $f \in S$ and $x \in \XX^2 \setminus A$ and that $\sum_{f \in S} e^{-\Delta_S(f)} \leq 1$.
There exists an universal constant $L_0 > 0$  such that if $L \geq L_0$,  the estimator $\hat{s}$ satisfies
\begin{eqnarray} \label{equationSelectPoints}
C \E \left[H^2 (s \1_A, \hat{s}) \right]  \leq  \E   \left[ \inf_{f \in S} \left\{ H^2(s \1_A, f)  + L \frac{\Delta_S(f)}{n} \right\}  \right] 
\end{eqnarray}
where $C$ is an universal positive constant.
\end{prop}

\subsection{A general model selection theorem.} 
We shall deduce from the above proposition a model selection theorem by choosing suitably $S$. To do so, we consider the following assumption.
\begin{hyp} \label{hypDensiteDesXiCasGeneral}
For all $i \in \{1,\dots,n-1\}$, $X_i$  admits a density  $\varphi_i$   with respect to some known measure $\nu$ such that $\nu(A_1) = 1$.
 Moreover,  there exists $\kappa$   such that $\varphi_i (x) \leq \kappa$ for all $x \in A_1$ and $i \in \{1,\dots,n-1\}$.
\end{hyp}
We define $\LL^2 (A, \nu \otimes \mu)$ the space of square integrable functions on $A$ with respect to the product measure $\nu \otimes \mu$, and we endow it with its natural distance 
$$d^2 (f,f') =  \int_{A}  \left(f (x,y) - f'(x,y) \right)^2 \d \nu (x) \d \mu (y) \quad \text{for all $f,f' \in \LL^2 (A, {\nu} \otimes \mu)$.}$$
Hereafter, a model $V$ is a (non-trivial) finite dimensional linear space of $\LL^2 (A, \nu \otimes \mu)$.

Let us explain how to obtain a model selection theorem when Assumption~\ref{hypDensiteDesXiCasGeneral} holds. Let~$\V$ be a collection of models $V$ and let $(\Delta(V))_{V \in \V}$ be a family of non-negative numbers such that $\sum_{V \in \V} e^{- \Delta(V)} \leq 1$.
For each model $V \in \V$, we consider an orthonormal basis $(f_1,\dots,f_{\dim V})$  of~$V$ and set
$$T_V = \left\{ \sum_{i=1}^{\dim V} \alpha_i f_i, \, \alpha_i \in \frac{2}{\sqrt{n \dim V}} \Z  \right\}.$$
We  deduce from Lemma~5 of~\cite{BirgeTEstimateurs} that the cardinal of $S_V = \left\{f_+^2 \1_A, \,f \in T_{V} , \, d (f,0) \leq 2\right\}$ is upper bounded by $|S_V| \leq \left(30 n \right)^{\dim V / 2}$.
We then use the above procedure with 
$S = \cup_{V \in \V} S_V$ and 
$$\Delta_S(f) =  \inf_{\substack{V \in \V \\  S_{V} \ni f}} \left\{\Delta(V)  + (\dim V) \log \left(30 n\right)/2  \right\} \quad \text{for all $f \in S$.}$$
This yields to an estimator $\hat{s}$ such that
\begin{eqnarray*}
C' \E \left[H^2 (s \1_A, \hat{s}) \right] \leq  \inf_{V \in \V} \left\{\kappa \left(  \inf_{\substack{f \in T_V \\ d (f,0) \leq 2}} d^2 \left( \restriction{\sqrt{s}}{A}, f \right) \right) + \frac{\Delta(V)  + \dim (V) \log n}{n} \right\} 
\end{eqnarray*}
where $C'$ is an universal positive constant.
Since $d (\restriction{\sqrt{s}}{A},0) \leq 1$,
$$\inf_{\substack{f \in T_V \\ d (f,0) \leq 2}} d^2 \left( \restriction{\sqrt{s}}{A}, f \right)  =  d^2 \left( \restriction{\sqrt{s}}{A}, T_V \right).$$
For all $f' \in V$, there exists $f \in T_{V}$ such that $d^2 (f,f') \leq n^{-1}$ and thus
$$d^2 \left(\restriction{\sqrt{s}}{A}, T_{V} \right) \leq 2 d^2 \left( \restriction{\sqrt{s}}{A}, V \right) + \frac{2}{n}.$$
Precisely, we have proved: 

\begin{thm} \label{ThmSelectGeneral}
Suppose that Assumption~\ref{hypDensiteDesXiCasGeneral} holds. Let $\V$ be an at most countable collection of models.
Let  $(\Delta (V))_{V \in \V}$ be a family of non-negative numbers such that $$\sum_{V \in \V} e^{- \Delta (V)} \leq 1.$$
There exists an estimator $\hat{s}$ such that
$$C \E \left[H^2 (s \1_A, \hat{s}) \right] \leq  \inf_{V \in \V} \left\{d^2 \left( \restriction{\sqrt{s}}{A}, V \right) + \frac{\Delta (V) + \dim (V) \log n}{n} \right\}$$
where $C > 0$ depends only on $\kappa$.
\end{thm}
The condition $\sum_{\V \in \boldsymbol{\V}} e^{- \Delta (V)} \leq 1$ can be interpreted as a (sub)probability on the collection~$\V$. The more complex the family $\V$, the larger the weights $\Delta(V)$.
When one can choose $\Delta(V)$ of order $\dim (V)$, which means that the family $\V$ of models does not contains too many models per dimension, the estimator~$\hat{s}$ achieves the best trade-off  (up to a constant) between the approximation  and the variance terms.

This theorem holds under an assumption that is very mild and weaker than those of~\cite{Lacour2007}, \cite{Akakpo2011} and \cite{Clemenccon2000}.~\cite{Birge2012} proved a general oracle inequality when there exist integers $k \geq 1$ and $l \geq 0$ and positive numbers $\rho, \varrho$ such that
$$\varrho \leq \frac{1}{k} \sum_{j=1}^k s^{(l + j)} (x,y) \leq \rho \quad \text{for all $x,y \in \XX$}$$
where the parameters $k, l, \varrho$ are known.  Our assumption is then satisfied for the Markov chain $(X_{l+1}, \dots, X_n)$ with $\nu = \mu$ and $\kappa = k \rho$. 

We shall consider subsets $\F \subset \LL^2 (A,  {\nu} \otimes \mu)$  corresponding to smoothness or structural assumptions on  $\restriction{\sqrt{s}}{A}$. For such an $\F$, we  associate a collection $\V$ and deduce from  Theorem~\ref{ThmSelectGeneral} a risk bound for the estimator $\hat{s}$ when $\restriction{\sqrt{s}}{A}$ belongs  to  $\F$. This set is a generic notation and will change from section to section. In the remaining part of this paper, we shall always choose $\XX^2 = \R^{2 d}$, $A = [0,1]^{2 d}$ and $\mu$ the Lebesgue measure.

\subsection{Smoothness assumptions.} We have introduced in Section~\ref{SectionRatesConvergence} the isotropic Besov spaces  $\mathscr{B}_{q}^{\sigma} (\L^p ([0,1]^{2 d}))$ where $\sigma \in (0,1)$. In this section, we consider  the    anisotropic Besov spaces   $\mathscr{B}_{q}^{\boldsymbol{\sigma}} (\L^p ([0,1]^{2 d}))$ where $\boldsymbol{\sigma} = ({\sigma}_1,\dots,{\sigma}_{2 d})$ belongs to $(0,+\infty)^{2 d}$.

Intuitively, a function $f$ on $[0,1]^{2 d}$ belongs to  $\mathscr{B}_{q}^{\boldsymbol{\sigma}} (\L^p ([0,1]^{2 d}))$ if, for all $j \in \{1,\dots,2 d\}$, and $x_1,\dots,x_{j-1},x_{j+1},\dots,x_{2 d} \in [0,1]$ the function
$$x_j  \mapsto f (x_1,\dots,x_{j-1},x_j,x_{j+1},\dots,x_{2 d})$$
belongs to  $\mathscr{B}_{q}^{{\sigma}_j} (\L^p ([0,1]))$.
In particular, for all $\sigma \in (0,+\infty)$, 
$$ \mathscr{B}_{q}^{\sigma} (\L^p ([0,1]^{2 d}))  = \mathscr{B}_{q}^{(\sigma,\dots,\sigma)} (\L^p ([0,1]^{2 d})).$$
A  definition of the anisotropic Besov spaces may be found in~\cite{Hochmuth2002} (for $d  = 1$) and in~\cite{Akakpo2009} (for larger values of $d$). We also consider the space $\mathcal{H}^{\boldsymbol{\sigma}}  ([0,1]^{2 d})$ of anisotropic Hölderian functions on $[0,1]^{2 d}$ with regularity~${\boldsymbol{\sigma}}$. A precise definition of this space may be found in Section~3.1.1 of~\cite{BaraudComposite} (among other references).

For all $\boldsymbol{\sigma} = (\sigma_1,\dots,\sigma_{2 d}) \in (0,+\infty)^{2 d}$, we denote by $\bar{\boldsymbol{\sigma}}$ the harmonic mean of $\boldsymbol{\sigma}$:
$$\frac{1}{\bar{\boldsymbol{\sigma}}} = \frac{1}{2 d} \sum_{i=1}^{2 d} \frac{1}{\sigma_i}.$$
We set for all $p \in (0,+\infty]$,
\begin{equation*}
 \mathscr{B}^{\boldsymbol{\sigma}} (\L^p ([0,1]^{2 d}))  =
  \begin{cases}
  \mathscr{B}^{\boldsymbol{\sigma}}_{\infty} (\L^p ([0,1]^{2 d})) & \text{if $p \in (0,1]$} \\
   \mathscr{B}^{\boldsymbol{\sigma}}_{p} (\L^p ([0,1]^{2 d}))  & \text{if $p \in (1,2)$} \\
  \mathscr{B}^{\boldsymbol{\sigma}}_{\infty} (\L^p ([0,1]^{2 d})) & \text{if $p \in [2,+\infty)$} \\   
   \mathcal{H}^{\boldsymbol{\sigma}} ([0,1]^{2 d})  & \text{if $p = \infty$} 
  \end{cases}
\end{equation*}
and denote by $|\cdot|_{p,\boldsymbol{\sigma}}$ the semi norm associated to the space $\mathscr{B}^{\boldsymbol{\sigma}} (\L^p ([0,1]^{2 d}))$. 

In this section, we are interesting in obtaining a bound risk when $\restriction{\sqrt{s}}{A}$ belongs to the  space
$$ \mathscr{B} ([0,1]^{2 d}) =  \bigcup_{p \in (0,+\infty]} \left( \bigcup_{\substack{\boldsymbol{\sigma} \in (0,+\infty)^d \\ \bar{\boldsymbol{\sigma}} >  2 d (1/p - 1/2)_+}}   \mathscr{B}^{\boldsymbol{\sigma}} (\L^p ([0,1]^{2 d}))\right).$$
Families of linear spaces possessing good approximation properties with respect to the elements of $\F = \mathscr{B} ([0,1]^{2 d})$ can be found in Theorem~1 of~\cite{Akakpo2012}. We then  deduce   from Theorem~\ref{ThmSelectGeneral}, 

\begin{cor}  \label{CorVitesseBesovAnisotropes}
Suppose that Assumption~\ref{hypDensiteDesXiCasGeneral} holds with $\XX = \R^{d}$, $A = [0,1]^{2 d}$ and with $\nu \otimes \mu$ the Lebesgue measure. 
There exists an  estimator $\hat{s}$ such that for all $\restriction{\sqrt{s}}{A} \in  \mathscr{B} ([0,1]^{2 d})$, 
\begin{eqnarray*}
C \E \left[H^2 \left(s \1_A, \hat{s} \right) \right] \leq   \left|\restriction{\sqrt{s}}{A}\right|_{p,\boldsymbol{\sigma}}^{2 d / (d +  \bar{\boldsymbol{\sigma}})} \left(\frac{\log n}{n} \right)^{\bar{\boldsymbol{\sigma}}/ (\bar{\boldsymbol{\sigma}} + d)}  + \frac{\log n}{n}
\end{eqnarray*}
where $p \in (0, +\infty]$, $\boldsymbol{\sigma} \in (0,+\infty)^{2 d}$, $\bar{\boldsymbol{\sigma}} >  2 d (1/p - 1/2)_+$ are such that $\restriction{\sqrt{s}}{A} \in \mathscr{B}^{\boldsymbol{\sigma}} (\L^p ([0,1]^{2 d}))$ and where $C > 0$ depends only on $\kappa$,$d$,$p$,$\boldsymbol{\sigma}$.

\end{cor} 
To our knowledge, the only statistical procedures that can adapt both to possible inhomogeneity  and anisotropy of $s$ are those of~\cite{Akakpo2011} and~\cite{Birge2012}. The losses  are different, but the rates are the same as ours (up to the logarithmic term). In view of our assumptions, we do not know if the logarithmic term can be avoided. 

In the following sections, we consider classes $\F$ corresponding to structural assumptions on~$\restriction{\sqrt{s}}{A}$.  More precisely, rates of convergence when the chain is autoregressive with constant conditional variance (respectively non constant conditional variance) are established in Section~\ref{SectionAutoRegressive} (respectively Section~\ref{SectionArch}).

\subsection{AR model.} \label{SectionAutoRegressive} In this section, we assume that  $X_{n+1} = g (X_n) + \varepsilon_n$
where $g$ is an unknown function and where the $\varepsilon_n$'s are  unobserved   identically distributed random variables.  Many papers are devoted to the estimation of the regression function $g$ and it is beyond the scope of this paper to make an historical review for this statistical problem. 

 For the sake of simplicity, one shall assume throughout this section that $\XX = \R$, $A = [0,1]^2$.
The transition density is of the form  $s(x,y) = \varphi (y - g (x))$
 where $\varphi$ is the density of~$\varepsilon_0$. 
Since~$g$ and~$\varphi$ are both unknown, this  suggests us to consider the  class
$$\F = \bigcup_{\sigma > 0}  \left\{f ,\, \exists \phi \in \mathcal{H}^{\sigma} (\R), \exists g \in \mathscr{B} ([0,1]),  \|g\|_{\infty} < \infty ,  \, \forall x,y \in [0,1] ,\; f(x,y) = \phi (y - g (x))   \right\}.$$
A family $\V$ of linear spaces possessing good approximation properties with respect to the functions of $\F$ can be built by using  Section~6.2 of~\cite{BaraudComposite}. Precisely, we prove the following.

\begin{cor}  \label{CorVitesseAutoRegressif}
Suppose that Assumption~\ref{hypDensiteDesXiCasGeneral} holds with $\XX = \R$, $A = [0,1]^2$ and with  $\nu \otimes \mu$ the Lebesgue measure on $\R^2$. 
Assume that  $\restriction{\sqrt{s}}{A}$ belongs to $\mathscr{F}$. Let   $\sigma > 0$, $p \in (0,+\infty]$, $\beta >  (1/p - 1/2)_+$ be any numbers and $\phi \in \mathcal{H}^{\sigma} (\R)$, $g \in \mathscr{B}^{\beta} (\L^{p} ([0,1]))$, $\|g\|_{\infty} < \infty$  be any functions such that
$$\sqrt{s (x,y)} = \phi (y - g (x)) \quad \text{for all $x,y \in [0,1]$.}$$
There exists two estimators $\hat{\phi} \geq 0$ and $\hat{g}$  such that the estimator~$\hat{s}$  defined by
 $$\hat{s}(x,y) =  \left(\hat{\phi}\left(y - \hat{g} (x)\right)  \right)^2 \1_{[0,1]^2} (x,y) \quad \text{for all $x,y \in \R$}$$
satisfies
\begin{eqnarray*}
C \E \left[H^2 \left(s \1_A, \hat{s} \right) \right] \leq   C'_1 \left(\frac{\log^2 n}{n}\right)^{\frac{2 \beta (\sigma \wedge 1)}{2 \beta (\sigma \wedge 1) + 1}}  + C'_2 \left(\frac{\log n}{n} \right)^{\frac{2 \sigma}{2 \sigma + 1}} 
\end{eqnarray*}
where $C > 0$ depends only on $\kappa$,$p$,$\sigma$,$\beta$, where $C'_1$ depends only on $p$,$\beta$,$\sigma$,$|g|_{p,\beta}$,$\|g\|_{\infty}$,$|\phi|_{\infty, \sigma \wedge 1}$ and where $C'_2$ depends only on $\sigma$,$\|g\|_{\infty}$,$|\phi|_{\infty,\sigma}$.
Moreover, the construction of the estimators $\hat{g}$, $\hat{\phi}$ depends only on the data $X_0,\dots,X_{n}$.
\end{cor} 
In particular, if $\phi$ is very smooth (says $\sigma \geq  \beta \vee 1$), the rate of convergence corresponds to the  rate of convergence for estimating $g$ only (up to a logarithmic term). 

It is interesting to compare the preceding rate to the one we would obtain under the pure smoothness assumption on $\restriction{\sqrt{s}}{A}$ but ignoring that $\restriction{\sqrt{s}}{A}$ belongs to $\F$.
To do so, we need to specify the regularity of $\restriction{\sqrt{s}}{A}$, knowing that of $\phi$ and $g$. This is the purpose of the following lemma.  

\begin{lemme} \label{LemmeRegulariteAutoRegressif}
Let $\sigma, \beta > 0$, and let us define 
\begin{equation*}
 \theta(\beta,\sigma) =
  \begin{cases}
   \beta \sigma  & \text{if $\beta,\sigma \leq 1$} \\
   \beta \wedge \sigma & \text{otherwise.} 
  \end{cases}
\end{equation*}
Let $\phi \in \mathcal{H}^{\sigma} (\R)$, $g \in \mathcal{H}^{\beta} ([0,1])$. The function $f$ defined by
$$f(x,y) = \phi (y - g (x)) \quad \text{for all $x,y \in [0,1]$},$$
belongs to $\mathcal{H}^{(\theta(\beta,\sigma), \sigma)} ([0,1]^2)$.

Moreover, for all $\sigma,\beta > 0$, there exist $\phi \in \mathcal{H}^{\sigma} (\R)$, $g \in \mathcal{H}^{\beta} ([0,1])$ such that the function  $f$ defined by
$$f(x,y) = \phi (y - g (x)) \quad \text{for all $x,y \in [0,1]$},$$
belongs to $\mathcal{H}^{(a,b)} ([0,1]^2)$ if and only if $a \leq  \theta (\beta,\sigma)$ and $b \leq \sigma$.
\end{lemme}
This result says that if  $\sqrt{s (x,y)} =  \phi (y - g (x)), $  with $\phi \in \mathcal{H}^{\sigma} (\R)$, $g \in \mathcal{H}^{\beta} ([0,1])$, then  $\sqrt{s}$ is 
Hölderian with regularity $(\theta(\beta,\sigma), \sigma)$ on $[0,1]^2$, and this regularity cannot be improved in general except in some particular situations. Under such a smoothness assumption, the rate of estimation we would get is  $(\log n/ n)^{2 \sigma \theta(\beta,\sigma) / (2 \sigma \theta(\beta,\sigma) + \theta(\beta,\sigma) + \sigma)}$. This rate is always slower than the rate obtained under the structural assumption.

\subsection{ARCH model.} \label{SectionArch} Throughout this section, we assume that $X_{n+1} = g_1 (X_n) + g_2 (X_n) \varepsilon_n$ where $g_1, g_2$ are unknown functions and where the $\varepsilon_n$'s are  unobserved   identically distributed  random variables.
The previous model corresponded to $g_2 = 1$. The problem of the estimation of the mean and variance functions $g_1$ and $g_2$ was considered in several papers and we refer to Section~1.2 of~\cite{Comte2002} for  bibliographical references.

For the sake of simplicity, one assumes that $\XX = \R$ and $A = [0,1]^2$.
If~$\varphi$ denotes the density of $\varepsilon_0$, the transition density $s$ is of the form 
\begin{eqnarray} \label{FormedeSmodeleARCH}
s(x,y) = |g_2(x)|^{-1} \varphi \left[ g_2^{-1} (x) \left(y - g_1 (x)\right) \right] \quad \text{for all $x,y \in \R$.}
\end{eqnarray}
We consider thus the class
\begin{eqnarray*}
\F &=& \bigcup_{\sigma > 0}  \left\{f ,\, \exists \phi \in \mathcal{H}^{\sigma} (\R), \exists v_1, v_2 \in \mathscr{B} ([0,1]), \, \|v_1\|_{\infty} < \infty, \, \|v_2\|_{\infty} < \infty ,\right.  \\
 & & \qquad \qquad \left.  \forall x,y \in [0,1] ,\,  f(x,y) = \sqrt{|v_2(x)|} \phi \left(v_2 (x) (y - v_1 (x))\right)\right\}
\end{eqnarray*}
and apply Theorem~\ref{ThmSelectGeneral} with a suitable collection $\V$ to obtain:
\begin{cor}  \label{CorVitesseARCH}
Suppose that Assumption~\ref{hypDensiteDesXiCasGeneral} holds with $\XX = \R$, $A = [0,1]^2$ and with  $\nu \otimes \mu$ the Lebesgue measure on $\R^2$. 
Assume that  $\restriction{\sqrt{s}}{A}$ belongs to $\mathscr{F}$. Let $\sigma > 0$,  $\phi \in \mathscr{B}^{\sigma} (\R)$ and for all  $i \in \{1,2\}$, let $p_i \in (0,+\infty]$, $\beta_i >  (1/p_i - 1/2)_+$, $v_i \in \mathscr{B}^{\beta_i} (\L^{p_i} ([0,1]))$, with  $\|v_i\|_{\infty} < \infty$ such that
$$\sqrt{s (x,y)} = \sqrt{|v_2(x)|} \phi \left(v_2 (x) (y - v_1 (x))\right) \quad \text{for all $x,y \in [0,1]$.}$$
Let $p_3 \in (0,+\infty]$ and $\beta_3 >  (1/p_3 - 1/2)_+$ be any numbers such that  $v_3 = \sqrt{|v_2|} \in \mathscr{B}^{\beta_3} (\L^{p_3} ([0,1]))$.
 There exists an estimator $\hat{s}$  such that 
\begin{eqnarray*}
C \E \left[H^2 \left(s, \hat{s} \right) \right] \leq   C'_1 \left(\frac{\log^2 n}{n} \right)^{\frac{2 \beta (\sigma \wedge 1)}{2 \beta  (\sigma \wedge 1) + 1}}  +  C'_2 \left( \frac{\log n}{n} \right)^{\frac{2 \sigma }{2 \sigma + 1}} 
\end{eqnarray*}
where  $\beta = \max(\beta_1,\beta_2,\beta_3)$.
The constant $C > 0$ depends only on $\kappa$,$\sigma$,$p_1$,$p_2$,$p_3$,$\beta_1$,$\beta_2$,$\beta_3$, $C'_1$ depends only on $\sigma$,$\|v_1\|_{\infty}$,$\|v_2\|_{\infty}$,$\|\varphi\|_{\infty}$,$|v_1|_{p_1,\beta_1}$,$|v_2|_{p_2,\beta_2}$,$|v_3|_{p_3,\beta_3}$,$|\varphi|_{\infty, \sigma \wedge 1}$ and $C'_2$ depends only on $\sigma$,$\|v_2\|_{\infty}$,$|\varphi|_{\infty,\sigma}$.
Moreover, the construction of the estimator $\hat{s}$ depends only on the data $X_0,\dots,X_{n}$.
\end{cor} 
If $s$ is of the form~{(\ref{FormedeSmodeleARCH})} with $\varphi$, $g_1$, $g_2$ smooth, in the sense that 
$\phi = \sqrt{\varphi} \in \mathcal{H}^{\sigma} (\R)$, $v_1 = g_1 \in \mathscr{B}^{\beta_1} (\L^{p_1} ([0,1]))$, $\|v_1\|_{\infty} < \infty$, $v_2 = g_2^{-1} \in \mathscr{B}^{\beta_2} (\L^{p_2} ([0,1]))$, $\|v_2\|_{\infty} < \infty$ and $v_3 = |g_2|^{-1/2} \in \mathscr{B}^{\beta_3} (\L^{p_3} ([0,1]))$, then $\restriction{\sqrt{s}}{A}$ belongs to $\F$. If $\phi$ is sufficiently smooth ($\sigma \geq \beta_1 \vee \beta_2 \vee \beta_3 \vee 1$), the rate becomes
 \begin{eqnarray*}
C'' \E \left[H^2 \left(s, \hat{s} \right) \right] \leq \max \left(\left(\frac{\log^2 n}{n} \right)^{\frac{2 \beta_1}{2 \beta_1   + 1}},  \left(\frac{\log^2 n}{n}\right)^{\frac{2 \beta_2 }{2 \beta_2  + 1}},  \left(\frac{\log^2 n}{n}\right)^{\frac{2 \beta_3 }{2 \beta_3  + 1}} \right).
\end{eqnarray*}
Up to a logarithmic term, the first term  corresponds to the bound we would get if we could estimate $g_1$ only. The two other terms correspond to the rate of estimation of $ g_2^{-1}$ and $ |g_2|^{-1/2}$ respectively (up to a logarithmic term).

 Note that if $\beta_2 \in (0,1)$, one can always choose $p_3 = 2 p_2$ (with $p_3 = \infty$ if $p_2 = \infty$), $\beta_3 = \beta_2 /2$, in which case the rate becomes
 \begin{eqnarray*} \label{ConvergenceModelARCH}
C'' \E \left[H^2 \left(s, \hat{s} \right) \right] \leq \max \left(\left(\frac{\log^2 n}{n} \right)^{\frac{2 \beta_1}{2 \beta_1   + 1}},  \left(\frac{\log^2 n}{n}\right)^{\frac{ \beta_2 }{ \beta_2   + 1}} \right).
\end{eqnarray*}
In some situations however, $\beta_3$ can be taken larger than $\beta_2$. 

As in the preceding section, we may use the lemma below to compare this rate with the one we would obtain under smoothness assumptions on $\restriction{\sqrt{s}}{A}$.
\begin{lemme} \label{LemmeVitesseArch}
Let for all $\sigma, \beta_1,\beta_2 > 0$,
\begin{equation*}
 \theta(\beta_1,\beta_2,\sigma) =
  \begin{cases}
   \left(2^{-1} (\beta_2 \wedge 1 ) \right) \wedge \sigma \beta_1 \wedge \sigma \beta_2  & \text{if $\sigma \leq 1$ and $\beta_1 \wedge \beta_2 \leq 1$} \\
   \left(2^{-1} (\beta_2 \wedge 1 ) \right) \wedge \sigma \wedge \beta_1 & \text{otherwise.} 
  \end{cases}
\end{equation*}
Let $\phi \in \mathcal{H}^{\sigma} (\R)$, $v_1 \in \mathcal{H}^{\beta_1} ([0,1])$, $v_2 \in \mathcal{H}^{\beta_2} ([0,1])$. The function $f$ defined by
$$ f(x,y) = \sqrt{|v_2(x)|} \phi \left(v_2 (x) (y - v_1 (x))\right) \quad \text{for all $x,y \in [0,1]$,}$$
belongs to $\mathcal{H}^{\left(\theta(\beta_1,\beta_2,\sigma), \sigma\right)} ([0,1]^2)$.

Moreover, there exist $\phi \in \mathcal{H}^{\sigma} (\R)$, $v_1 \in \mathcal{H}^{\beta_1} ([0,1])$, $v_2 \in \mathcal{H}^{\beta_2} ([0,1])$ such that the function $f$ defined by
$$ f(x,y) = \sqrt{|v_2(x)|} \phi \left(v_2 (x) (y - v_1 (x))\right) \quad \text{for all $x,y \in [0,1]$,}$$
belongs to $\mathcal{H}^{\left(a,b\right)} ([0,1]^2)$
if and only if $a \leq \theta(\beta_1,\beta_2,\sigma)$ and $b\leq \sigma$.
\end{lemme}
This proposition says that if $\sqrt{s (x,y)} = \sqrt{|v_2(x)|} \phi \left(v_2 (x) (y - v_1 (x))\right),$ with $\phi \in \mathcal{H}^{\sigma} (\R)$, $v_1 \in \mathcal{H}^{\beta_1} ([0,1])$, $v_2 \in \mathcal{H}^{\beta_2}  ([0,1])$, $\restriction{\sqrt{s}}{A}$ belongs to  $\mathcal{H}^{\left(\theta(\beta_1,\beta_2,\sigma), \sigma\right)}  ([0,1]^2)$ and the regularity index of this space cannot be increased in general. 
By Corollary~\ref{CorVitesseBesovAnisotropes}, we would get a rate of order
$\left(\log n/n \right)^{2 \theta(\beta_1,\beta_2,\sigma) \sigma/(2 \theta(\beta_1,\beta_2,\sigma) \sigma + \theta(\beta_1,\beta_2,\sigma) + \sigma)},$
which is slower than the one given by Corollary~\ref{CorVitesseARCH}.

\section{Appendix: implementation of the first procedure.} \label{SectionAlgorithm}
In this section, we explain how to construct in practice the estimator of the first procedure. This will lead to the proposition below.
\begin{prop} \label{PropositionComplexity}
For all $L > 0$, $\ell \in \N^{\star}$, the estimator $\hat{s} = \hat{s} (L,\ell)$ of Section~\ref{SectionProcedure} can be built in less than $C \left( n \ell  d + \ell 4^{(\ell + 1) d} \right) $ operations 
where~$C$ is an universal constant.
\end{prop}
We set for all $K \in \cup_{m \in \M_{\ell}} m$,
$$ \hat{s}_K = \frac{\sum_{i=0}^{n-1} \1_K (X_i,X_{i+1})}{\sum_{i=0}^{n-1} \int_{\XX} \1_K (X_i,x) \d \mu (x)} \1_K,$$ for all $K' \in \cup_{m \in \M_{\ell}} m$,
$$F_K (K') = \alpha H^2 \left(\hat{s}_K \1_{K'}, \hat{s}_{K'}\1_{K}\right) +  T \left(\hat{s}_K \1_{K'}, \hat{s}_{K'}\1_{K}\right),$$ and for all $m' \in \M_{\ell}$,
\begin{eqnarray*} 
 \gamma_K (m') = \left( \sum_{K' \in m' }  F_K (K') \right)  -   \pen (m'  \vee K).
\end{eqnarray*}
We shall find for each cube $K \in \cup_{m \in \M_{\ell}} m$, a partition $m'_K \in \M_{\ell}$ such that
\begin{eqnarray} \label{relationgamma}
  \gamma_K (m'_K) = \sup_{m' \in \M_{\ell}} \gamma_K (m').
  \end{eqnarray}
We shall compute then
\begin{eqnarray} \label{relationgamma2}
\min_{m \in \M_{\ell} }  \gamma(m) = \min_{m \in \M_{\ell} } \left\{ \left(\sum_{K \in m} \gamma_K (m_K') \right) + 2 \pen(m)\right\}.
\end{eqnarray}
We shall find $m'_K$ by using a slight adaptation of the procedure of~\cite{Blanchard2004}. Computing~{(\ref{relationgamma2})} is similar. The algorithm we propose is based on the one-to-one correspondence between $\M_{\ell}$ and the set $\mathcal{T}_{\ell}$ of $4^d$-ary trees with depth smaller than $\ell$. 
\begin{lemme} \label{LemmeLienArbreEtPartition}
There exists $\psi_{\ell}$ a one-to-one map between $\M_{\ell}$ and $\mathcal{T}_{\ell}$ such that for all $m \in \M_{\ell}$, $\psi_{\ell} (m)$ is a tree whose leaves  correspond to the elements of the partition~$m$.
\end{lemme}
The construction of this map may for instance be deduced from Section 3.2.4 of~\cite{BaraudBirgeHistogramme}.

We need to introduce some notations. For each tree $T \in \mathcal{T}_{\ell}$ and bin $K''$ of $T$, we denote by $T(K'')$ the subtree of $T$ rooted in~$K''$. The set of leaves of $T (K'')$ is denoted by $\mathcal{L} (T (K''))$. We set~$R(K'')$ the tree reduced to its root $K''$ (\emph{i.e},  $\mathcal{L} (R (K'')) = \{K''\}$).
For all cube $K \in \cup_{m \in \M_{\ell}} m$, we  set  $$\mathcal{L} (T (K'')) \vee K = \left\{K' \cap K , \, K' \in \mathcal{L} (T (K'')), \, K' \cap K \neq \emptyset \right\}$$
and we define the function~$\mathcal{E}$ by
$$\mathcal{E} (T (K'')) = - |\mathcal{L} (T (K'')) \vee K| + \sum_{K' \in \mathcal{L} (T (K'')) } F_K (K').$$
The key point is that computing~{(\ref{relationgamma})} amounts to finding $T^{\star}$ such that $$\mathcal{E} (T^{\star} ([0,1]^{2 d})) = \sup_{T \in \mathcal{T}_{\ell}} \mathcal{E} (T ([0,1]^{2 d}))$$
since $m'_K = \psi^{-1}_{\ell} (T^{\star}).$ 

We now take advantage of the additivity of the function $\mathcal{E}$: if $T (K'')$ is not reduced to its root, and if $K''_1,\dots,K''_{4^d}$ are the cubes of $\cup_{m \in \M_{\ell}} m$ such that $K'' = \cup_{i=1}^{4^d} K''_i$, then,
\begin{eqnarray} \label{RelAdditionFonctionE}
\mathcal{E} (T (K'')) = \sum_{i=1}^{4^d} \mathcal{E} (T(K''_i)).
\end{eqnarray}
For all cube $K'' \in \cup_{m \in \M_{\ell}} m$, let $T^{\star} (K'')$ be a tree (rooted in $K''$) such that
$$\mathcal{E} (T^{\star} (K''))  = \sup_{T \in \mathcal{T}_{\ell},\, T \ni K''} \mathcal{E} (T (K'')).$$
Remark that if $K'' \cap K = \emptyset$, this supremum is equal to $0$, in which case $T^{\star} (K'')$  will  always stand for $R(K'')$. In general, we deduce from~{(\ref{RelAdditionFonctionE})} that
\begin{eqnarray} \label{RelationSomme}
\mathcal{E} (T^{\star} (K'')) = \max \left( \mathcal{E} (R (K'')), \sum_{i=1}^{4^d} \mathcal{E} (T^{\star} (K''_i))  \right).
\end{eqnarray}
Calculating~{(\ref{relationgamma})} can  thus be completed in that way: we start with the sets $K'' \in \cup_{m \in \M_{\ell} \setminus \M_{\ell-1}} m$ with $K'' \cap K \neq \emptyset$ for which the optimal local trees are reduced to their roots.
 By using relation~{(\ref{RelationSomme})} we  find the optimal local trees $T^{\star} (K'')$ when $K'' \in \cup_{\M_{\ell-1} \setminus \M_{\ell}} m$,  $K'' \cap K \neq \emptyset$. Proceeding recursively like this  yields to the optimal tree $T^{\star} = T^{\star} ([0,1]^{2 d})$.

\section{Proofs}

\subsection{Proof of Proposition~\ref{prophisto}.}
Let us introduce the piecewise constant function
\begin{eqnarray} \label{RelationEstimateurSBarrem}
\bar{s}_{m} = \sum_{K \in m} \frac{\sum_{i=0}^{n-1} \E \left[\1_K (X_i,X_{i+1}) \mid X_i \right]}{\sum_{i=0}^{n-1} \int_{\XX} \1_K (X_i,x) \d \mu (x)} \1_K.
\end{eqnarray}
By using the triangular inequality we can decompose the risk of $\hat{s}_{m}$ as follows:
$$ \E \left[H^2(s \1_A, \hat{s}_{m})\right]  \leq \left(1 + \frac{2 + \log 2}{2} \right) \E\left[H^2(s\1_A, \bar{s}_{m})  \right]+ \left(1 + \frac{2}{2 + \log 2} \right) \E \left[H^2(\bar{s}_{m}, \hat{s}_{m}) \right].$$
The first term can be bounded from above by $(4 + \log 2 )\, \E \left[H^2 (s\1_A, V_{m}) \right]$ thanks to Lemma~2 of~\cite{BaraudBirgeHistogramme}. For the second term, we begin to define for $K \in m$  the random variable  $$B_K =  \left( \sqrt{\sum_{i=0}^{n-1} \1_{K} (X_i,X_{i+1}) } -  \sqrt{ \sum_{i=0}^{n-1} \E \left[\1_{K} (X_i,X_{i+1}) \mid X_i \right] } \right)^2.$$
Since $2 n H^2 (\hat{s}_{m}, \bar{s}_{m})  = \sum_{K \in m} B_K$, we shall bound from above the terms $\E[B_K]$.
For this purpose, we introduce the stopping time
$$T = \inf \left\{i \geq 0, \; \E \left[\1_{K} (X_{i}, X_{i+1} ) \mid X_{i} \right] \geq \frac{1}{2 n} \right\} \wedge (n-1) $$
with respect to the filtration $\mathcal{F}_n = \sigma(X_0,\dots,X_n)$ generated by the random variables $X_0,\dots,X_n$. We set $\varepsilon = 1+ \log 2 +2 \log n $  and  we decompose $\E[B_K ]$ as follows 
\begin{eqnarray} \label{eqPreuveHistogramme}
\E \left[B_K \right] &\leq& \left(1 + \varepsilon \right) \E \left[ \left( \sqrt{\sum_{i=0}^{T-1} \1_{K} (X_i,X_{i+1}) } -  \sqrt{ \sum_{i=0}^{T-1} \E \left[\1_{K} (X_i,X_{i+1}) \mid X_i \right] } \right)^2 \right] \nonumber \\
& & \quad + \left(1 + \varepsilon^{-1} \right) \E \left[ \left( \sqrt{\sum_{i=T}^{n-1} \1_{K} (X_i,X_{i+1}) } -  \sqrt{ \sum_{i=T}^{n-1} \E \left[\1_{K} (X_i,X_{i+1}) \mid X_i \right] } \right)^2 \right] \nonumber \\
 &\leq& 2 \left(1+\varepsilon\right)  \E \left[\sum_{i=0}^{T-1}  \E \left[\1_{K} (X_i,X_{i+1}) \mid X_i \right]  \right] \nonumber \\
& & \quad + \left(1+\varepsilon^{-1}\right) \E \left[\frac{ \left(\sum_{i=T}^{n-1} \left(\1_{K} (X_i,X_{i+1}) - \E \left[\1_{K} (X_i,X_{i+1}) \mid X_i \right] \right) \right)^2}{\sum_{i=T}^{n-1} \E \left[\1_{K} (X_i,X_{i+1}) \mid X_i \right] }  \right].  
\end{eqnarray}
Yet,
$$\E \left[\sum_{i=0}^{T-1}  \E \left[\1_{K} (X_i,X_{i+1}) \mid X_i \right]  \right] \leq 1/2, $$
and we control the second term of the right-hand side of inequality~{(\ref{eqPreuveHistogramme})}, by using the claims below.

\begin{Claim} \label{Claim1}
For all $K \in m$, $j  \in \{0,\dots,n\}$, and  $\A' \in \mathcal{F}_j = \sigma(X_0,\dots,X_j)$,
\begin{eqnarray*}
\E \left[\frac{ \left(\sum_{i=j}^{n-1} \left(\1_{K} (X_i,X_{i+1}) - \E \left[\1_{K} (X_i,X_{i+1}) \mid X_i \right] \right) \right)^2}{\sum_{i=j}^{n-1} \E \left[\1_{K} (X_i,X_{i+1}) \mid X_i \right] } \1_{\A'} \right] \leq \sum_{k=j}^{n-1} \E \left[ \frac{\E \left[\1_{K} (X_k,X_{k+1}) \mid X_k \right]}{\sum_{i=j}^{k} \E \left[\1_{K} (X_i,X_{i+1}) \mid X_i \right]} \1_{\A'}   \right].
\end{eqnarray*}
\end{Claim}
\begin{proof} [Proof of Claim~\ref{Claim1}.]
Let us define the random variables
$$Y_{n-1} (K) = \sum_{i=j}^{n-1} \left(\1_{K} (X_i,X_{i+1}) - \E \left[\1_{K} (X_i,X_{i+1}) \mid X_i \right] \right) \; \text{and} \; Z_n (K) = \frac{Y_{n-1}^2 (K) }{\sum_{i=j}^{n-1} \E \left[\1_{K} (X_i,X_{i+1}) \mid X_i \right]}. $$
We have
\begin{eqnarray*}
\E \left[Z_{n+1} (K) \mid \mathcal{F}_n \right] &=& \frac{\E \left(\left[Y_{n-1}(K) + \left(\1_{K} (X_n,X_{n+1}) - \E \left[\1_{K} (X_n,X_{n+1}) \mid X_n \right] \right) \right]^2 \mid \mathcal{F}_n \right) }{\sum_{i=j}^{n} \E \left[\1_{K} (X_i,X_{i+1}) \mid X_i \right]} \\
&=& \frac{Y_{n-1}^2(K) + \var \left( \1_{K} (X_n,X_{n+1}) \mid X_n \right) }{\sum_{i=j}^{n} \E \left[\1_{K} (X_i,X_{i+1}) \mid X_i \right]} \\
&\leq& Z_n (K) +  \frac{\E \left[\1_{K} (X_n,X_{n+1}) \mid X_n \right]}{\sum_{i=j}^{n} \E \left[\1_{K} (X_i,X_{i+1}) \mid X_i \right]}.
\end{eqnarray*}
Thus, since $\A'$ is also $\mathcal{F}_{n}$-measurable,
\begin{eqnarray*}
\E \left[Z_{n+1} (K) \1_{\A'}  \right]  \leq \E \left[Z_n (K)  \1_{\A'} \right] +  \E \left[\frac{\E \left[\1_{K} (X_n,X_{n+1}) \mid X_n \right]}{\sum_{i=j}^{n} \E \left[\1_{K} (X_i,X_{i+1}) \mid X_i \right]} \1_{\A'} \right].
\end{eqnarray*}
The result ensues  from induction.
\end{proof}

\begin{Claim} \label{Claim2}
For all sequence $(u_n)_{n \geq 0}$ in $[0,1]$, and $j \geq 0$ such that $u_j \neq 0$, 
$$ \sum_{k=j}^{n-1} \frac{u_k}{\sum_{i=j}^{k} u_i} \leq 1 + \log n  - \log u_j.$$
\end{Claim}
\begin{proof} [Proof of Claim~\ref{Claim2}.]
Let $f$ be any non-negative continuous function such that $u_k = \int_{k}^{k+1} f(t) \d t$ whatever $k \in \N$. Let $F$ be the primitive of $f$ such that $F(j) = 0$. Then,
\begin{eqnarray*}
\sum_{k=j}^{n-1} \frac{u_k}{\sum_{i=j}^{k} u_i} &\leq& 1 + \sum_{k=j+1}^{n-1}  \int_{k}^{k+1} \frac{f(t)}{F(k+1)} \d t \\
&\leq& 1 + \sum_{k=j+1}^{n-1}  \int_{k}^{k+1} \frac{f(t)}{F(t)} \d t \\
&\leq& 1 + \log F(n) - \log F(j+1) \\
&\leq& 1 + \log \left(\sum_{k=j}^{n-1} u_k \right) - \log u_j.
\end{eqnarray*}
\end{proof}

By using Claim~\ref{Claim1} with $\A' = [T = j]$,
\begin{eqnarray*}
& & \!\!\!\!\!\!\!\!\!\! \!\! \!\! \!\! \!\!  \E \left[\frac{ \left(\sum_{i=T}^{n-1} \left(\1_{K} (X_i,X_{i+1}) - \E \left[\1_{K} (X_i,X_{i+1}) \mid X_i \right] \right) \right)^2}{\sum_{i=T}^{n-1} \E \left[\1_{K} (X_i,X_{i+1}) \mid X_i \right] }   \right] 
\leq  \sum_{j=0}^{n-2}   \E \left( \sum_{k=j}^{n-1} \frac{\E \left[\1_{K} (X_k,X_{k+1}) \mid X_k \right]}{\sum_{i=j}^{k} \E \left[\1_{K} (X_i,X_{i+1}) \mid X_i \right]}   \1_{T = j} \right)  \\
& & \qquad \qquad  \qquad \qquad \qquad \qquad\qquad  + \E \left[\frac{ \left( \1_{K} (X_{n-1},X_n) - \E \left[\1_{K} (X_{n-1},X_{n}) \mid X_{n-1} \right]  \right)^2}{\E \left[\1_{K} (X_{n-1},X_{n}) \mid X_{n-1} \right] } \1_{T = n -1 }  \right].
\end{eqnarray*}
Now,
\begin{eqnarray*}
\E \left[\frac{ \left( \1_{K} (X_{n-1},X_n) - \E \left[\1_{K} (X_{n-1},X_{n}) \mid X_{n-1} \right]  \right)^2}{\E \left[\1_{K} (X_{n-1},X_{n}) \mid X_{n-1} \right] } \1_{T = n -1 }  \right] \!\!\!\!\! &=& \!\!\!\!\!  \E \left[\frac{  \var \left( \1_{K} (X_{n-1},X_n) \mid X_{n-1}   \right)}{\E \left[\1_{K} (X_{n-1},X_{n}) \mid X_{n-1} \right] } \1_{T = n -1 } \right] \\
&\leq& \P(T = n -1) .
\end{eqnarray*}
We then use Claim~\ref{Claim2} with $u_k =  \E \left[\1_{K} (X_k,X_{k+1}) \mid X_k \right]$ to derive 
\begin{eqnarray*}
 \sum_{j=0}^{n-2}   \E \left( \sum_{k=j}^{n-1} \frac{\E \left[\1_{K} (X_k,X_{k+1}) \mid X_k \right]}{\sum_{i=j}^{k} \E \left[\1_{K} (X_i,X_{i+1}) \mid X_i \right]}   \1_{T = j} \right)   &\leq& \sum_{j=0}^{n-2} \E \left[ (1 + \log 2 + 2 \log n ) \1_{T = j} \right] \\
 &\leq& \left(1 + \log 2 + 2 \log n \right) \P(T \neq n -1).
\end{eqnarray*}
Finally,
$\E\left[B_K \right] \leq 4 + 2 \log 2 +4 \log n$ and hence $$\E \left[H^2(\bar{s}_{m}, \hat{s}_{m}) \right] \leq \frac{2 + \log 2 + 2\log n}{ n} |m| ,$$
which  concludes the proof.
\qed

\subsection{Proof of Theorem~\ref{CorTheoremSelectionHisto}.} \label{SectionProofTheorem1} 
When $\ell \leq n$, the result ensues from the following theorem whose proof  is delayed to Section~\ref{SectionProofTheorem2}. In the theorem below, the constant $L_0 = 90$ can  easily be  improved but it seems to be difficult to obtain the value $L_0 = 0.03$ used in practice.
\begin{thm} \label{ThmSelectPractique2}
For all $L \geq 90$ and $1 \leq \ell \leq n$, the estimator  $\hat{s} = \hat{s} (L,\ell)$  satisfies 
 $$\forall \xi  > 0, \quad \P \left[ C H^2 \left(s \1_{A}, \hat{s}  \right) \geq  \inf_{m \in \M_{\ell}} \left(H^2 \left(s  \1_{A}, \hat{s}_m \right) + \pen(m) \right) + \xi \right] \leq 3 e^{-n \xi}$$
where $C$ is an universal positive constant. 
\end{thm}
By integrating the inequality above, there exists  $C' > 0$ such that
$$C' \E \left[H^2 \left(s \1_{A}, \hat{s} \right) \right] \leq  \inf_{m \in \M_{\ell}} \left\{ \E  \left[H^2 \left(s \1_{A}, \hat{s}_m \right) \right] + \pen(m) \right\}$$
and the conclusion follows from  Proposition~\ref{prophisto}.

When $\ell$ is larger than $n$, we use the lemma below whose proof is postponed to Section~\ref{SectionProofLemmaProcedureStationnaire}.
\begin{lemme} \label{LemmeProcedureStationaire}
For all $L \geq 15$ and $\ell \geq n+1$, $\hat{s} (L,\ell) = \hat{s} (L,n)$ and $\hat{s} (L,\infty) = \hat{s} (L,n)$.
\end{lemme}
Consequently, if $\ell \geq n + 1$ or $\ell = \infty$,
$$C' \E \left[H^2 \left(s \1_{A}, \hat{s} \right) \right] \leq  \inf_{m \in \M_{n}} \left\{ \E  \left[H^2 \left(s \1_{A}, V_m \right) \right] + \pen(m) \right\}.$$
Let $m^{\star} \in \M_{\ell}$ such that
$$2 \inf_{m \in \M_{\ell}} \left\{  \E  \left[H^2 \left(s \1_{A}, V_m \right) \right] + \pen(m) \right\} \geq   \E  \left[H^2 \left(s \1_{A}, V_{m^{\star}} \right) \right] + \pen(m^{\star}).$$
Since
$$\inf_{m \in \M_{\ell}} \left\{  \E  \left[H^2 \left(s \1_{A}, V_m \right) \right] + \pen(m) \right\}\leq \frac{1}{2} + L \frac{\log n}{n},$$
we deduce  $L |m^{\star}| \log (n) / n \leq  1 + 2 L \log(n) / n$ and thus $|m^{\star}| \leq 2 + n / ( L \log n) \leq n$.
Remark now that the cardinal of a partition $m \in \M_{\ell} \setminus \M_{n }$ can be lower bounded by 
$$|m| \geq 4^d + (4^d - 1) n \geq n + 1.$$
Consequently, $m^{\star} \in \M_n$ and hence,
$$\inf_{m \in \M_{n}} \left\{  \E  \left[H^2 \left(s \1_{A}, V_m \right) \right] + \pen(m) \right\} \leq 2 \inf_{m \in \M_{\ell}} \left\{  \E  \left[H^2 \left(s \1_{A}, V_m \right) \right] + \pen(m) \right\}$$
which completes the proof.
\qed

\subsection{Proof of Theorem~\ref{ThmSelectPractique2}.} \label{SectionProofTheorem2}
The proof of this theorem requires the two following lemmas whose proofs are postponed to Sections~\ref{SectionPreuveLemmeControlGamma} and~\ref{SectionPreuveLemmeProprietesDuTest}.

\begin{lemme} \label{LemmeControlGamma}
For all $m \in \M_{\ell}$, there exists a deterministic set $S_m$ containing $\hat{s}_m$ such that
\begin{eqnarray*}
\gamma_1 (m) = \sup_{m' \in \M_{\ell}}  \left\{  \alpha H^2 \left(\hat{s}_m,  \hat{s}_{m'} \right) +  T \left(\hat{s}_m, \hat{s}_{m'} \right) -  \pen(m')    \right\} +   \pen(m) 
\end{eqnarray*}
and
\begin{eqnarray*}
\gamma_2 (m) = \sup_{\substack{m' \in \M_{\ell} \\ f' \in S_{m'}}}  \left\{  \alpha H^2 \left(\hat{s}_m,  f' \right) +  T \left(\hat{s}_m, f'\right) -   \pen(m')    \right\} +   2 \pen (m) 
\end{eqnarray*}
satisfy
$$\gamma_1(m) \leq \gamma (m) \leq \gamma_2 (m).$$
\end{lemme}

\begin{lemme} \label{LemmeProprietesDuTest}
Set  $\varepsilon =  \left(2+3 \sqrt{2}\right)/8$. Under assumptions of Theorem~\ref{ThmSelectPractique2},  for all $\xi > 0$, there exists an event $\Omega_{\xi}$ such that $\P(\Omega_{\xi}) \geq 1 - 3 e^{-n \xi}$ and on which,
\begin{eqnarray} \label{RelationLemmeProprieteTest}
 & & \quad  \text{for all partition  $m \in \M_{\ell}$,}  \\
& & \sup_{\substack{m' \in \M_{\ell} \\ f' \in S_{m'}}} \left\{  \left(1 - \varepsilon \right) H^2 \left(s\1_A, f'  \right) + T\left(\hat{s}_m, f'  \right) - \pen(m') \right\} \leq\left(1 + \varepsilon \right) H^2 \left(s\1_A, \hat{s}_m \right) + \pen(m) + 22 \xi \nonumber
\end{eqnarray}
where $S_{m'}$ is defined in Lemma~\ref{LemmeControlGamma}.
\end{lemme}

\begin{proof} [Proof of Theorem~\ref{ThmSelectPractique2}.]
On $\Omega_{\xi}$, for all $m \in \M_{\ell}$,
\begin{eqnarray*}
\sup_{\substack{m' \in \M_{\ell} \\ f' \in S_{m'}}} \left\{\left(1 - \varepsilon \right) H^2 \left(s\1_A, f'  \right) + T\left(\hat{s}_{m},f'  \right)   -   \pen(m' ) \right\} \leq\left(1 + \varepsilon \right) H^2 \left(s\1_A, \hat{s}_{m} \right) +   \pen(m)  + 22 \xi.
\end{eqnarray*}
If $ T\left(\hat{s}_{m},\hat{s}_{\hat{m}} \right) +    \pen (m) -   \pen(\hat{m} ) \geq 0$, 
\begin{eqnarray*}
\alpha H^2 \left(s\1_A, \hat{s}_{\hat{m}} \right) &\leq& (1 - \varepsilon) H^2 \left(s\1_A, \hat{s}_{\hat{m}} \right) +   T\left(\hat{s}_{m},\hat{s}_{\hat{m}} \right) -  \pen (\hat{m} ) +     \pen (m) \\
&\leq& \left(1 + \varepsilon \right) H^2 \left(s\1_A, \hat{s}_{m} \right)  +  2   \pen (m)  + 22 \xi
\end{eqnarray*}
since $\alpha \leq 1 - \varepsilon$ and since $\hat{s}_{\hat{m}}$ belongs to $\cup_{m' \in \M_{\ell}} S_{m'}$.

If $ T\left(\hat{s}_{m},\hat{s}_{\hat{m}} \right) +   \pen(m) -  \pen (\hat{m} ) < 0$, 
\begin{eqnarray*}
\alpha H^2 \left(\hat{s}_m, \hat{s}_{\hat{m}} \right) &\leq& \alpha  H^2 \left(\hat{s}_{\hat{m}}, \hat{s}_m  \right) +   T\left(\hat{s}_{\hat{m}}, \hat{s}_{m} \right)  -    \pen(m) +  \pen(\hat{m} )  \\
&\leq&  \sup_{m' \in \M_{\ell}} \left\{\alpha H^2 \left(\hat{s}_{\hat{m}}, \hat{s}_{m'}  \right) +   T\left(\hat{s}_{\hat{m}}, \hat{s}_{m'} \right)  -     \pen(m')   \right\}  +   \pen(\hat{m} ) \\
&\leq& \gamma_1 (\hat{m}).
\end{eqnarray*}
Consequently, by Lemma~\ref{LemmeControlGamma}, 
$$\gamma_1 (\hat{m}) \leq \gamma(m) + \frac{1}{n} \leq \gamma_2(m) +\frac{1}{n},$$
which implies that
\begin{eqnarray*}
\alpha H^2 \left(\hat{s}_m, \hat{s}_{\hat{m}} \right)
\leq  \sup_{\substack{m' \in \M_{\ell} \\ f' \in S_{m'}}}  \left\{  \alpha H^2 \left(\hat{s}_m,  f' \right) +  T \left(\hat{s}_m, f'\right) - \pen(m')    \right\} + 2 \pen(m)  + \frac{1}{n}. 
\end{eqnarray*}
With $\upsilon = (1 - \varepsilon)/ \alpha - 1 > 0$,
\begin{eqnarray*}
\alpha H^2 \left(\hat{s}_m, \hat{s}_{\hat{m}} \right)  &\leq&  \left(1 + \upsilon^{-1} \right)  H^2 \left(\hat{s}_m,  s \1_A \right)   \\ 
& &  \qquad + \sup_{\substack{m' \in \M_{\ell} \\ f' \in S_{m'}}}  \left\{ (1  - \varepsilon)   H^2 \left(s\1_A,  f' \right) +    T \left(\hat{s}_m, f'\right) -   \pen (m')    \right\} + 2 \pen (m) + \frac{1}{n}  \\
&\leq&    \left(1 + \upsilon^{-1} \right)  H^2 \left(\hat{s}_m,  s \1_A \right)  +  \left[ \left(1 + \varepsilon \right) H^2 \left(s\1_A, \hat{s}_{m} \right) +    \pen(m)  +  22 \xi \right]  + 2 \pen (m)   +  \frac{1}{n} \\
&\leq&   \left(2 + \varepsilon + \upsilon^{-1} \right) H^2 \left(\hat{s}_m,  s \1_A \right)   + 3 \pen (m)  +  22 \xi +  \frac{1}{n}.
\end{eqnarray*}
This leads to
\begin{eqnarray*}
\alpha H^2 \left(s\1_A, \hat{s}_{\hat{m}} \right)  &\leq& 2 \alpha  H^2 \left(s\1_A, \hat{s}_m\right) + 2 \alpha H^2 \left(\hat{s}_m, \hat{s}_{\hat{m}} \right) \\
&\leq& 2 \left(2  + \alpha +  \varepsilon + \upsilon^{-1} \right)  H^2 \left(\hat{s}_m,  s \1_A \right)  + 6 \pen(m) + 44 \xi  + \frac{2 }{n}.
\end{eqnarray*}
Finally, we have proved that there exists $C > 0$, such that,  with probability larger than $1 - 3 e^{-n \xi}$, for all $m \in \M_{\ell}$,
\begin{eqnarray*}
C H^2 \left(s\1_A, \hat{s}_{\hat{m}} \right)  \leq    H^2 \left(\hat{s}_m,  s \1_A\right)  +  \pen(m) + \xi.
\end{eqnarray*}
This concludes the proof.
\end{proof}

\subsubsection{Proof of Lemma~\ref{LemmeControlGamma}.} \label{SectionPreuveLemmeControlGamma}
\begin{Claim} \label{ClaimModels}
Let, for all $K \in \cup_{m \in \M_{\ell}} m$, $K_1,\dots,K_{l}$ be the cubes of $ \cup_{m \in \M_{\ell}} m$ such that $K \subset K_i$ for all $i \in \{1,\dots,l\}$.
For all $i \in \{1,\dots,l\}$, let $I_i$ and $J_i$ be the subsets of $[0,1]^d$ such that $K_i = I_i \times J_i$.
Set
$$S_K = \bigcup_{i=1}^l \left\{ \frac{a}{b  \mu(J_i)} \1_{K} , \, a \in \{0,\dots, n\}, \, b \in \{1 , \dots, n\} \right\}$$ 
with the convention $ a / 0 = 0$ whatever $a \in \{0,\dots, n\}$. Then $|S_K| \leq \ell n(n+1) $, and
\begin{eqnarray} \label{eqTK}
\forall K' \in \cup_{m \in \M_{\ell}} m, \, K \subset K', \quad \frac{\sum_{i=0}^{n-1} \1_{K'} (X_i,X_{i+1})}{\sum_{i=0}^{n-1} \int_{\XX} \1_{K'} (X_i,x) \d \mu (x)}  \1_K \in S_K.
\end{eqnarray}
\end{Claim}
We then define $$S_m = \left\{ \sum_{K \in m} f_K , \, f_K \in S_K \right\}$$
where $S_K$ is given by the claim above and introduce the random set 
$$\hat{S}_m = \left\{ \sum_{K \in m} \hat{s}_{m_K} \1_K ,\, \forall K \in m, \,  m_K \in \M_{\ell} \right\}.$$
For all  $\hat{f} \in \hat{S}_m$, we denote by $m_K (\hat{f})$ any partition of $\M_{\ell}$ such that
$$\hat{f}  (x)  = \hat{s}_{m_K (\hat{f})} (x) \quad \text{for all $x \in K$}$$
and consider  the partition $$m(\hat{f}) = \bigcup_{K \in m} (m_K (\hat{f}) \vee K ).$$
By definition,
\begin{eqnarray*}
\gamma (m) = 2 \pen(m) + \sum_{K \in m} \sup_{m'_K \in \M_{\ell}} \left[ \alpha H^2 \left(\hat{s}_m\1_{K}, \hat{s}_{m'_K }\1_{K}\right) +  T \left(\hat{s}_m \1_{K},\hat{s}_{m'_K }\1_{K} \right) -  \pen(m'_K \vee K  )   \right]
\end{eqnarray*}
and thus
\begin{eqnarray*}
\gamma (m)  &=& \sup_{\substack{ \hat{f} \in \hat{S}_{m } }}  \left\{\sum_{K \in m}  \left[ \alpha H^2 \left(\hat{s}_m \1_{K}, \hat{f} \1_{K}\right) +  T \left(\hat{s}_m\1_{K}, \hat{f} \1_{K} \right) -  \pen (m_K (\hat{f}) \vee K)   \right] \right\} +  2 \pen(m) \\
&=& \sup_{\substack{ \hat{f} \in \hat{S}_{m } }}  \left\{  \alpha H^2 \left(\hat{s}_m,  \hat{f} \right) +  T \left(\hat{s}_m, \hat{f} \right) -  \pen (m(\hat{f}))    \right\} +  2 \pen(m). 
\end{eqnarray*}
Now, for all $m,m' \in \M_{\ell}$, the estimator $\hat{s}_{m'}$, belongs to $\hat{S}_{m}$ with 
$$m(\hat{s}_{m'}) = \left\{ K \cap K', \, (K,K') \in m \times m', \, K \cap K' \neq \emptyset\right\},$$
 which leads to
\begin{eqnarray*}
\gamma (m)  \geq \sup_{m' \in \M_{\ell}}  \left\{  \alpha H^2 \left(\hat{s}_m,  \hat{s}_{m'} \right) +  T \left(\hat{s}_m, \hat{s}_{m'} \right) - \pen (m(\hat{s}_{m'}))    \right\} +    2 \pen(m).
\end{eqnarray*}
Since $m(\hat{s}_{m'}) \subset  m \cup m'$, $|m(\hat{s}_{m'})| \leq  |m| + |m'|$ and $\gamma_1(m) \leq \gamma(m)$.

Let us now prove the inequality $\gamma \leq \gamma_2$. A function $\hat{f} \in \hat{S}_m$ is constant on each set of the partition~$m(\hat{f})$.
For  $K' \in m(\hat{f})$, there exist $K \in m$, $K'' \in m_K (\hat{f})$ with $K' \subset K''$ such that
 $$\hat{f} \1_{K'} = \frac{\sum_{i=0}^{n-1} \1_{K''} (X_i,X_{i+1})}{\sum_{i=0}^{n-1} \int_{\XX} \1_{K''} (X_i,x) \d \mu (x)}  \1_{K'}.$$ 
By relation~{(\ref{eqTK})}, $\hat{f} \1_{K'}  \in S_{K'}$ and thus $\hat{f} = \sum_{K' \in m(\hat{s})} \hat{f} \1_{K'}$  belongs  to $S_{m(\hat{f})}$. Consequently, $ \hat{S}_{m } \subset \cup_{m' \in \M_{\ell}} S_{m'}$ and the conclusion follows.
\qed

\subsubsection{Proof of Lemma~\ref{LemmeProprietesDuTest}.}\label{SectionPreuveLemmeProprietesDuTest}
We start with the claim below.
\begin{Claim}
Let $\psi$ be the function defined on $[0,+\infty)^2$ by 
$$ \psi(x,y) = \frac{1}{\sqrt{2}} \frac{\sqrt{y} - \sqrt{x}}{\sqrt{x + y}} \quad \text{for all $x,y \in [0,+\infty)$}$$
with the convention $0/0 = 0$.

Let, for all $f,f' \in \L_+^1 (\XX^2, M)$, with support included in $A$, $Z(f,f')$ be the random variable defined by
\begin{eqnarray*}
 Z(f, f')  = \frac{1}{n}\sum_{i=0}^{n-1}  \left(  \psi \left(f (X_i, X_{i+1}), f'(X_i, X_{i+1}) \right)  - \int_{\XX}   \psi \left(f (X_i, y), f' (X_i, y) \right) (s \1_A) (X_i, y) \d \mu(y) \right).
\end{eqnarray*}
Then,
\begin{eqnarray} \label{RelationTEtZ}
\left(1 - \frac{1}{\sqrt{2}} \right) H^2 \left(s\1_A, f'  \right) + T\left(f,f'  \right)  \leq\left(1 + \frac{1}{\sqrt{2}} \right) H^2 \left(s\1_A, f \right) + Z \left(f, f'  \right)
\end{eqnarray}
and
\begin{eqnarray} \label{RelationTEtZ2}
\frac{1}{n} \sum_{i=0}^{n-1}  \int_{\XX} \psi^2 \left(f(X_i,y), f'(X_i,y) \right) \d \mu (y) \leq 3  \left( H^2 (s\1_A, f) + H^2(s\1_A,f') \right).
\end{eqnarray}
\end{Claim}
\begin{proof}
These inequalities can be obtained by using the same arguments as those used in the proofs of Propositions~2 and 3 of~\cite{BaraudMesure}.
\end{proof}
We shall prove~{(\ref{RelationLemmeProprieteTest})} by applying the following concentration inequality to the random variable $Z \left(f, f'  \right)$.

\begin{Claim} \label{ClaimConcentrationInequality}
For all $i \leq n-1$, let $\mathcal{F}_i$ be the $\sigma$-field generated by the random variables $X_j$ for $j \in \{0,\dots, i\}.$ 
Let $f_1,\dots,f_n \in \LL^1 (\XX^2,M)$ such that there exists  $b \in \R$ with $\sup_{x \in \XX^2} |f_i(x)| \leq b$ for all $i \in \{0,\dots, n-1\}$. Set
$$S_n =  \sum_{i=0}^{n-1} \left( f_i (X_i,X_{i+1}) - \E \left[ f_i (X_i, X_{i+1}) \mid \mathcal{F}_i \right] \right)$$
and $$V_n = \sum_{i=0}^{n-1} \E \left[ f_i^2 (X_i, X_{i+1}) \mid \mathcal{F}_i \right].$$
Then, for all $\beta > b$ and $x > 0$
$$\P \left[ S_n \geq \frac{V_n}{2 (\beta - b)} + \beta x \right] \leq e^{-x}.$$
\end{Claim}

\begin{proof}
 By setting $a^{-1} = 2 (\beta - b)$,
\begin{eqnarray*}
\log \P \left[ S_n \geq a V_n + \beta x \right] &\leq& -  x + \log \E \left[ \exp \left(  \beta^{-1}  S_n - a \beta^{-1}  V_n \right) \right] \\
&\leq& -  x + \log \E \left[   \exp \left(\beta^{-1}  S_{n-1} - a  \beta^{-1}  V_n \right)  \E \left[ \exp \left( \beta^{-1} (  S_n -  S_{n-1})  \right)  \mid \mathcal{F}_{n-1} \right] \right]. 
\end{eqnarray*}
By using Bernstein inequality (Proposition~2.9 of~\cite{Massart2003}),
$$\E \left[ \exp \left( \beta^{-1} (  S_n -  S_{n-1})  \right)  \mid \mathcal{F}_{n-1} \right]  \leq \exp \left(\frac{ \beta^{-2} ( V_n - V_{n-1})}{2 (1 - \beta^{-1} b )} \right)$$
and thus
\begin{eqnarray*}
\log \P \left[ S_n \geq a V_n + \beta x \right]  \leq  -  x + \log \E \left[  \exp  \left(\beta^{-1}  S_{n-1} - a  \beta^{-1} V_{n-1} \right)    \right]. 
\end{eqnarray*}
The result follows by induction.
\end{proof}

\begin{proof} [Proof of Lemma~\ref{LemmeProprietesDuTest}.]
Set $z = (1-1/\sqrt{2})/4$,  $\beta = (3/z + \sqrt{2})/2$ and for all $\xi > 0$, 
$$\Omega_{\xi}  = \left\{ \sup_{\substack{(f,f') \in S_m \times S_{m'} \\ (m, m') \in \M_{\ell}^2}}  \frac{Z (f , f') }{z \left(H^2 (f, s \1_A) + H^2 ( f', s \1_A) \right)  + \pen(m) + \pen(m') + \beta \xi} < 1 \right\}.$$
On $\Omega_{\xi} $, for all $m,m' \in \M_{\ell}$, $(f,f') \in S_m \times S_{m'}$,
$$Z (f , f') \leq z \left(H^2 (f, s \1_A) + H^2 ( f', s \1_A) \right)  + \pen(m) + \pen(m') + \beta \xi$$
and~{(\ref{RelationLemmeProprieteTest})} derives from~{(\ref{RelationTEtZ})} (with $\varepsilon = 1/\sqrt{2} + z$).

It remains to prove that $\P(\Omega_{\xi}^c) \leq 3 e^{-n \xi}$. We have
\begin{eqnarray*}
\P\left(\Omega_{\xi}^c \right) \leq \sum_{\substack{(f,f') \in S_m \times S_{m'} \\ (m, m') \in \M_{\ell}^2}} \P \left[ Z (f, f') \geq z \left[ H^2 (s\1_A, f) + H^2(s\1_A,f') \right]   +   \pen(m) + \pen(m') + \beta \xi \right].
\end{eqnarray*}
We apply the concentration inequality given by Claim~\ref{ClaimConcentrationInequality} with $f_i = \psi \left( f, f' \right)$, $S_n =  n Z (f, f') $ and by using relation~{(\ref{RelationTEtZ2})}, 
$$V_n = \sum_{i=0}^{n-1} \E \left[ f_i^2 (X_i, X_{i+1}) \mid \mathcal{F}_i \right] \leq 3 n \left( H^2 (s\1_A, f) + H^2(s\1_A,f') \right).$$
We obtain for all $x > 0$,
$$\P \left[ Z (f, f') \geq \frac{3}{\sqrt{2} (\beta \sqrt{2}-1)} \left[ H^2 (s \1_A, f) + H^2(s\1_A,f') \right] + \beta \frac{x}{n} \right] \leq e^{-x}.$$
Note that $z = 3/ (\sqrt{2} (\beta \sqrt{2}-1))$. By using the inequality above with 
$$\beta \frac{ x}{n} =  \pen(m) + \pen(m') + \beta \xi$$
 we deduce that
\begin{eqnarray*}
\P\left(\Omega_{\xi}^c  \right) \leq \sum_{\substack{(f,f') \in S_m \times S_{m'} \\ (m, m') \in \M_{\ell}^2}} e^{- n \left( \beta^{-1} \pen(m) + \beta^{-1}  \pen(m') +  \xi \right)}. 
\end{eqnarray*}
Now, by Claim~\ref{ClaimModels}, since $\ell \leq n$, $\log |S_m| \leq 3 |m|  \log (n+1)$ and thus $\beta^{-1} \pen(m) \geq  (|m| + \log |S_m|)/ n$ for all $m \in \M_{\ell}$. Consequently,
\begin{eqnarray*}
\P\left(\Omega_{\xi}^c  \right)
&\leq& \sum_{\substack{(f,f') \in S_m \times S_{m'} \\ (m, m') \in \M_{\ell}^2}} e^{- \left(|m| + \log |S_m| + |m'| + \log |S_{m'}| + n \xi \right)} \\
&\leq&  \left(\sum_{ m \in \M_{\ell}} e^{-|m|} \right)^2 e^{-n \xi}.
\end{eqnarray*}
The conclusion follows from the inequality $\sum_{ m \in \M_{\ell}} e^{-|m|} \leq \sqrt{3}$ (see Section 3.2.4 of~\cite{BaraudBirgeHistogramme}).
\end{proof}

\subsection{Proof of Lemma~\ref{LemmeProcedureStationaire}.} \label{SectionProofLemmaProcedureStationnaire}
The lemma follows from the two claims below.
\begin{Claim} \label{ClaimProcedureStationnaire}
Let for each $m_1,m_2 \in \M_{\infty}$ and $K \in m_1$,
\begin{eqnarray*}
\gamma_K ({m_1}, {m_2})    =   \alpha H^2 \left(\hat{s}_{m_1} \1_{K}, \hat{s}_{m_2} \1_{K}\right) +  T \left(\hat{s}_{m_1}\1_{K},\hat{s}_{m_2}\1_{K } \right)   -   \pen ({m_2} \vee K).
\end{eqnarray*}
Then, for all $\ell \in \N^{\star}$, $\ell \geq  n + 1$, $m_1 \in \M_{\infty}$, $K \in m_1$,
$$\sup_{m_2 \in \M_{\ell}} \gamma_K (m_1,m_2)  = \sup_{m_2 \in \M_n} \gamma_K (m_1,m_2)$$
and thus
$$\sup_{m_2 \in \M_{\infty}} \gamma_K (m_1,m_2)  = \sup_{m_2 \in \M_n} \gamma_K (m_1,m_2).$$
\end{Claim}
\begin{proof}
Let $m_2^{\star} \in \M_{\ell}$ such that $\gamma_K (m_1,m_2^{\star}) = \sup_{m_2 \in \M_{\ell}} \gamma_K (m_1,m_2)$.
In Section~\ref{SectionExPartition}, we have defined the collection $\M_{\ell}$ of partitions of $[0,1]^{2 d}$. Likewise, by using the algorithm of~\cite{DeVore1990}, we define the collection $\M_{\ell} (K)$ of partitions of $K$.
Note that  $m_2^{\star} \vee K$ belongs to $\M_{\ell} (K)$.
Since
$ H^2 \left(\hat{s}_m \1_{K}, \hat{s}_{m'}\1_{K}\right) \leq 1$ and  $|T \left(\hat{s}_m\1_{K},\hat{s}_{m'}\1_{K } \right) | \leq 2$, we have
$$\gamma_K (m_1, m_2^{\star}) \leq 3 - L  \frac{|m_2^{\star} \vee K| \log n}{n}.$$
Remark that
$$\gamma_K (m_1, m_2^{\star}) \geq  \gamma_K \big(m_1, \{[0,1]^{2 d}\} \big) \geq -2 - L \frac{\log n}{n}$$
which leads to
$$|m_2^{\star} \vee K| \leq 1 + \frac{5 n}{L \log n } \leq n.$$
This implies that $m_2^{\star} \vee K$ belongs to $\M_n (K)$.
There exists  $m_2^{\bullet} \in \M_n$ such that $m_2^{\bullet} \vee K  = m_2^{\star} \vee K$ and hence
$\gamma_K (m_1,m_2^{\bullet})= \gamma_K (m_1,m_2^{\star})$
which concludes the proof.
\end{proof}

\begin{Claim}
Set for all $m \in \M_{\infty}$ and $K \in m$,
$$\gamma_K (m) = \sup_{m_2 \in \M_{\ell}} \gamma_K (m,m_2).$$
Then,
 $\gamma(m) = 2 \pen(m) + \sum_{K \in m} \gamma_K(m)$
and for all $\ell \in \N^{\star}$,  $\ell \geq n + 1$,
$$\inf_{m \in \M_{\ell}} \gamma(m) = \inf_{m \in \M_{n}} \gamma(m)$$
and thus $$\inf_{m \in \M_{\infty}} \gamma(m) = \inf_{m \in \M_{n}} \gamma(m).$$
\end{Claim}
\begin{proof}
Let $m^{\star} \in \M_{\ell}$ such that $ \inf_{m \in \M_{\ell}} \gamma(m) = \gamma(m^{\star})$.
By Lemma~\ref{LemmeControlGamma},
\begin{eqnarray*}
\gamma(m^{\star}) &\geq&  \sup_{m' \in \M_{\ell}}  \left\{  \alpha H^2 \left(\hat{s}_m,  \hat{s}_{m'} \right) +  T \left(\hat{s}_m, \hat{s}_{m'} \right) -  \pen(m')    \right\} +  L \frac{|m^{\star}|\log n}{n} \\
&\geq&    \left(  -2  -  L \frac{\log n}{n}   \right) +  L \frac{|m^{\star}|\log n}{n} \\
&\geq&   -2   +  L \frac{\left(|m^{\star}| - 1\right)\log n}{n}.
\end{eqnarray*}
Now, 
$$\gamma(m^{\star}) \leq \gamma(\{[0,1]^{2 d}\}) \leq 2 L \frac{\log n }{n}  + 3$$
which implies that 
$$ \left|m^{\star}\right| \leq 3 + \frac{5 n}{L \log n} \leq n$$
and thus $m^{\star} \in \M_{n}$.
\end{proof}

\subsection{Proof of Theorem~\ref{CorTheoremSelectionHistoDistanceDeterministe}.}
Consider the regular partition  $m_{ref}$ of $[0,1]^{2 d}$ into cubes with side length~$2^{-\ell}$, that is $$m_{ref} = \left\{K_{\ell,\mathbf{l}}, \mathbf{l} = (k, \dots, k) , \, k \in \{1,\dots,2^{\ell}\} \right\}$$
where $K_{\ell,\mathbf{l}}$ is defined in Section~\ref{SectionChoixM}. 
For all partition $m \in \M_{\ell}$,  $V_{m} \subset V_{m_{ref}}$. Set
$$\Omega_{eq} = \left[\forall g_1, g_2 \in V_{m_{ref}} , \; h^2 (g_1,g_2) \leq 11 H^2(g_1,g_2)  \right]$$
and define $\bar{s}_m$ an element of $V_m$ such that $h^2(s \1_A, \bar{s}_m) = h^2(s \1_A,V_m)$. 

For all $m \in \mathcal{M}_{\ell}$,
\begin{eqnarray*}
\E \left[h^2 \left(s \1_A, \hat{s}_{\hat{m}} \right) \right] &\leq&   \E \left[h^2 \left(s\1_A, \hat{s}_{\hat{m}} \right)  \1_{\Omega_{eq}} \right]  + \E \left[h^2 \left(s\1_A, \hat{s}_{\hat{m}}  \right) \1_{\Omega_{eq}^c} \right] \\
&\leq& 2 \E \left[h^2 \left(s\1_A, \bar{s}_{m} \right)  \1_{\Omega_{eq}} \right] + 2 \E \left[h^2 \left( \bar{s}_{m}, \hat{s}_{\hat{m}} \right)  \1_{\Omega_{eq}} \right]  + \E \left[h^2 \left(s\1_A, \hat{s}_{\hat{m}}  \right) \1_{\Omega_{eq}^c} \right]  \\
&\leq& 2 \E \left[h^2 \left(s\1_A, \bar{s}_{m} \right)  \1_{\Omega_{eq}} \right] + 22 \E \left[H^2 \left( \bar{s}_{m}, \hat{s}_{\hat{m}} \right)  \1_{\Omega_{eq}} \right]  + \E \left[h^2 \left(s\1_A, \hat{s}_{\hat{m}}  \right) \1_{\Omega_{eq}^c} \right]  \\
&\leq& 2 \E \left[h^2 \left(s\1_A, \bar{s}_{m} \right)  \1_{\Omega_{eq}} \right] + 44 \E \left[H^2 \left(s\1_A, \bar{s}_{m} \right)  \1_{\Omega_{eq}} \right]  + 44 \E \left[H^2 \left(\hat{s}_{\hat{m}} ,s \1_A \right)  \1_{\Omega_{eq}} \right]  \\
 & & \quad + \E \left[h^2 \left(s \1_A, \hat{s}_{\hat{m}}  \right) \1_{\Omega_{eq}^c} \right].
\end{eqnarray*}
Now, $h^2 \left(s \1_A, \bar{s}_{m} \right)  = \E [H^2 \left(s \1_A, \bar{s}_{m} \right) ]  = h^2(s \1_A,V_m)$ and
\begin{eqnarray*}
\E \left[h^2 \left(s \1_A, \hat{s}_{\hat{m}}  \right) \1_{\Omega_{eq}^c} \right] &\leq&  2 \E\left[ \left(h^2(s ,0) +  h^2 (\hat{s}_{\hat{m}},0)  \right) \1_{\Omega_{eq}^c} \right] \\
&\leq& \E \left[\left(1 + 2 \sup_{m \in \M_{\ell}}  h^2 (\hat{s}_{m},0) \right) \1_{\Omega_{eq}^c} \right].
\end{eqnarray*}
Let for all $K \in m$, $I_K$ and $J_K$ be the subsets of $[0,1]^{d}$ such that $K = I_K \times J_K$. Then,
\begin{eqnarray*}
2 h^2 (\hat{s}_{m},0) =  \sum_{K \in m}  \frac{\sum_{i=0}^{n-1} \1_{I_K} (X_i) \1_{J_K} (X_{i+1})}{\sum_{i=0}^{n-1} \1_{I_K} (X_i) } \int_{I_K} \varphi(x)  \d x 
\leq |m|.
\end{eqnarray*}
Since $m \subset m_{ref}$, $|m| \leq |m_{ref}| = 4^{\ell d}$ and thus,
\begin{eqnarray*}
C' \E \left[h^2 \left(s \1_A, \hat{s}_{\hat{m}} \right) \right] \leq  \inf_{m \in \M_{\ell}} \left\{ h^2 \left(s\1_A, V_m  \right)  + \pen(m) \right\} + 4^{\ell d} \P(\Omega_{eq}^c)
\end{eqnarray*}
for some universal constant $C' > 0$.

We now bound from above the term $\P(\Omega_{eq}^c)$. We denote by  $\I_{ref}$ the regular partition of $[0,1]^{d}$ into cubes with side length~$2^{- \ell}$. Remark that
\begin{eqnarray*}
\P(\Omega_{eq}^c) &\leq& \P \left[\exists I \in \I_{ref}, \, \P(X_1 \in I) \geq \frac{11}{n}  \sum_{i=0}^{n-1} \1_{I} (X_i) \right] \\
&\leq& 2^{\ell d} \sup_{I \in \I_{ref}} \P \left[\frac{1}{n} \sum_{i=0}^{n-1} \left(\1_{I} (X_i) - \P(X_i \in I) \right) \leq -   \frac{10}{11} \P(X_1 \in I) \right].
\end{eqnarray*}
We use the following Bennett-type inequality for $\beta$-mixing random variables (with $f = -\1_{I}$, $v =  \P(X_1 \in I)$, $c = 0$, $\xi = 10/11 \P(X_1 \in I)$). 
\begin{prop} \label{InegaliteConcentration}
Let $(X_i)_{i \geq 1}$ be a stationary sequence of random vectors with values in $\R^d$, and let $f$ be a real-valued function on $\R^d$ bounded from above by~$c \geq 0$ such that $v = \E \left[f(X_i)^2 \right] < \infty$.

Then, for all $q \in \{1,\dots, n\}$ and $\xi > 0$,
$$\P \left( \frac{1}{n} \sum_{i=1}^n (f(X_i) - \E(f(X_i))  > \xi \right) \leq 2 \exp \left( -  \frac{n \xi^2}{8 q \left( v + c \xi / 6 \right) } \right) + 3 n \beta_q / q. $$ 
\end{prop}
We then have for all $I \in \I_{ref}$,
\begin{eqnarray*}
\P \left[\frac{1}{n} \sum_{i=0}^{n-1} \left(\1_{I} (X_i) - \P(X_i \in I) \right) \leq -   \frac{10}{11} \P(X_1 \in I) \right]  \!\!\!\! &\leq&  \!\!\!\!  3 \inf_{1 \leq q \leq n} \left\{ \exp \left( -  \frac{25 n  \P(X_1 \in I)}{242 q   } \right) +  n \beta_q /q \right\} \\
\!\!\!\! &\leq& \!\!\!\!  3 \inf_{1 \leq q \leq n} \left\{ \exp \left( -  \frac{n \kappa_0 }{10  q  2^{\ell d} } \right) +  n \beta_q /q  \right\}
\end{eqnarray*}
which concludes the proof. \qed

\begin{proof} [Proof of Proposition~\ref{InegaliteConcentration}.]
Let  $l$ be the smallest integer larger than  $n / (2 q)$. We derive from Berbee's lemma and more precisely from~\cite{Viennet1997} (page 484) that there exist $X_1^{\star}, \dots, X_{2 l q}^{\star}$ such that
\begin{itemize}
\item For $j = 1,\dots, l$, the random vectors
$$\mathbf{X}_{j,1} = (X_{2 (j-1)q + 1}, \dots, X_{2 (j-1)q + q}) \quad \text{and} \quad \mathbf{X}^{\star}_{j,1} = (X_{2 (j-1)q + 1}^{\star}, \dots, X_{2 (j-1)q + q}^{\star}) $$
have the same distribution, and so have the random vectors
$$\mathbf{X}_{j,2} = (X_{2 (j-1)q + q + 1}, \dots, X_{2 j q}) \quad \text{and} \quad \mathbf{X}^{\star}_{j,2} = (X_{2 (j-1)q + q + 1}^{\star}, \dots, X_{2 j q}^{\star}).$$
\item The random vectors $\mathbf{X}_{1,1}^{\star},\dots,\mathbf{X}_{l,1}^{\star}$ are independent.  The random vectors $\mathbf{X}_{1,2}^{\star},\dots,\mathbf{X}_{l,2}^{\star}$ are also independent.
\item The event $$\Omega^{\star} = \bigcap_{1 \leq j \leq l}\left( \left[ \mathbf{X}_{j,1} \neq \mathbf{X}_{j,1}^{\star} \right] \cap \left[ \mathbf{X}_{j,2} \neq \mathbf{X}_{j,2}^{\star} \right]\right)$$ 
satisfies $\P \left[ (\Omega^{\star} )^c \right] \leq 2 l \beta_q.$
\end{itemize}
We set $g_i(x) = f(x)$ if $i \leq n$ and $g_i(x) = 0$ otherwise. 
For  $j \in \{1,\dots, l \}$, we set
$$g'_{j,1} (x_1,\dots, x_{q}) = \sum_{i=1}^{q} g_{2 (j-1)q + i}(x_i) \quad \text{and} \quad g'_{j,2} (x_1,\dots, x_{q}) = \sum_{i=1}^{q} g_{2 (j-1)q + q + i}(x_i) .$$
 Then,
\begin{eqnarray*}
\P \left[ \left( \frac{1}{n} \sum_{i=1}^{n} \left(g_i(X_i) - \E\left[g_i(X_i)\right]\right)  > \xi \right) \cap \Omega^{\star} \right] &\leq& \P \left( \sum_{j=1}^{l} \left(g'_{j,1}(\mathbf{X}_{j,1}^{\star}) - \E \left[g'_{j,1}(\mathbf{X}_{j,1}^{\star})\right]\right)  >  n \xi/2 \right)  \\
& &  + \P \left(  \sum_{j=1}^{l} \left(g'_{j,2}(\mathbf{X}_{j,2}^{\star}) - \E \left[g'_{j,2}(\mathbf{X}_{j,2}^{\star})\right]\right) >  n \xi/2 \right) \\
&\leq& 2 \exp \left( - \frac{n^2 \xi^2}{8  q  \left(n v + c n \xi / 6 \right) } \right)
\end{eqnarray*}
by using Proposition~2.8 of~\cite{Massart2003}.
\end{proof}

\subsection{Proof of Corollary~\ref{CorVitesseBesovIsotropes1}.}
The corollary ensues from the claim below  and Theorem~2 of~\cite{BaraudBirgeHistogramme}. 

\begin{Claim} \label{ClaimInfimum}
Under Assumption~\ref{hypDensiteDesXi1}, for all $\ell \in \N^{\star}$ such that $2^{\ell d} \geq n$, 
$$ \inf_{m \in \M_{\ell}} \left\{  d_2^2 \left(\restriction{\sqrt{s}}{A}, V_m \right) + \frac{ |m| \log n}{n} \right\}  \leq 4   \inf_{m \in \M_{\infty}} \left\{  d_2^2 \left(\restriction{\sqrt{s}}{A}, V_m \right) +  \frac{|m| \log n}{n} \right\}.$$
\end{Claim}

\begin{proof}
For all partition $m \in \M_{\infty}$ and cube $K \in m$, we denote by~$I_K$ and $J_K$  the cubes of $[0,1]^{d}$ such that $K = I_K \times J_K$ and set
$$\bar{s}_{m} = \sum_{K \in m}  \frac{\int_{K} s (x,y)  \d x \d y}{\mu \otimes \mu (K)} \1_K.$$
In this paper, $d_2$ stands for the standard euclidean distance of $\LL^2([0,1]^{2 d}, \mu \otimes \mu)$. In this proof, we make a small abuse of notations by denoting by $d_2$ the standard euclidean distance of $\LL^2(\R^{2 d}, \mu \otimes \mu)$.
 
Let $m^{\star}$ be a partition of $\M_{\infty}$ such that 
$$2 \inf_{m \in \M_{\infty}} \left\{  d^2_2 \left(\sqrt{s} \1_{A}, V_m \right) +  \frac{|m| \log n}{n} \right\} \geq  d^2_2 \left(\sqrt{s} \1_{A}, V_{m^{\star}} \right) +  \frac{|m^{\star}| \log n}{n}.$$
Let $\mathcal{C}$ be the collection $\mathcal{C} = \{K \in m^{\star}, \, \mu(I_K) \geq 2^{- \ell d} \}$ and let 
 $m^{\bullet}$ be a partition of $\M_{\ell}$ containing~$\mathcal{C}$ such that
 $$|m^{\bullet}| = \inf \{|m| , \, m \in \M_{\ell} \; \text{such that} \; m \ni \mathcal{C} \}.$$
 Let $A^{\bullet}$ be the set defined by $ A^{\bullet} = \cup_{K \in m^{\bullet}} K$ and 
 $V_{m^{\bullet}}^{\bullet} = \left\{f \1_{A^{\bullet}}, \, f \in V_{m^{\bullet}}  \right\}.$

 We have,
 \begin{eqnarray*}
  d^2_2 \left(\sqrt{s} \1_A, V_{m^{\bullet}}  \right) \leq  d^2_2 \left(\sqrt{s} \1_{A^{\bullet}}, V_{m^{\bullet}}^{\bullet}  \right) +  d^2_2 \left(\sqrt{s}  \1_{A \cap (A^{\bullet})^c}, 0  \right)
 \end{eqnarray*}
 and 
 $$d^2_2 \left(\sqrt{s}  \1_{A^{\bullet}}, V_{m^{\bullet}}^{\bullet}  \right) \leq  d^2_2 \left(\sqrt{s}  \1_{A^{\bullet}}, \sqrt{\bar{s}_{m^{\bullet}}} \1_{A^{\bullet}} \right) \leq d^2_2 \left(\sqrt{s}  \1_{A}, \sqrt{\bar{s}_{m^{\star}}} \right).$$
By using Lemma~2 of~\cite{BaraudBirgeHistogramme}, $ d^2_2 \left(\sqrt{s}  \1_{A}, \sqrt{\bar{s}_{m^{\star}}} \right) \leq 2d^2_2 \left(\sqrt{s}  \1_{A}, V_{m^{\star}} \right) $ which shows that
 \begin{eqnarray*}
  d^2_2 \left(\sqrt{s}  \1_A, V_{m^{\bullet}}  \right)  \leq   2 d^2_2 \left(\sqrt{s}  \1_{A}, V_{m^{\star}} \right) +  d^2_2 \left(\sqrt{s}  \1_{A \cap (A^{\bullet})^c}, 0  \right).
 \end{eqnarray*}
Now,
\begin{eqnarray*}
 d^2_2 \left(\sqrt{s} \1_{A \cap  (A^{\bullet})^c}, 0 \right) &\leq& \sum_{K \in m^{\star} \setminus \mathcal{C}} \int_{I_K}  \left( \int_{\R^d} s(x,y)  \d y \right)  \d x \\
&\leq& \sum_{K \in m^{\star} \setminus \mathcal{C}} \mu (I_K) \leq   2^{- \ell d} |m^{\star}|.
\end{eqnarray*}
Since $|m^{\bullet}| \leq |m^{\star}|$, we have
  \begin{eqnarray*}
  d^2_2 \left(\sqrt{s}  \1_A, V_{m^{\bullet}}  \right) + \frac{|m^{\bullet}| \log n}{n} \leq  2 d^2_2 \left(\sqrt{s}  \1_{A}, V_{m^{\star}} \right) +  \frac{(1 + \log n) |m^{\star}|}{n}
 \end{eqnarray*}
which  proves the claim.
\end{proof}

 \subsection{Rates of convergences for $h$.} \label{SectionPreuvesCasHDeterministe}
 We prove the result only for geometrically $\beta$-mixing chains (the proof for arithmetically $\beta$-mixing chains being similar). We use the claim below whose proof is the same  than the one of Claim~\ref{ClaimInfimum}.
\begin{Claim} 
Under Assumption~\ref{hypDensiteDesXi1}, for all $\ell \in \N^{\star}$ such that $2^{\ell d} \geq  n / \log^3 n$, 
$$ \inf_{m \in \M_{\ell}} \left\{  h^2 \left(s \1_{A}, V_m \right) + \frac{ |m| \log n}{n} \right\}  \leq 4  \inf_{m \in \M_{\infty}} \left\{  h^2 \left(s \1_A, V_m \right) +  \frac{|m| \log^3 n}{n} \right\}.$$
\end{Claim}
By using this claim and Theorem~1 of~\cite{Akakpo2012}, 
\begin{eqnarray}  \label{equationvitessedeterministe}
C \E \left[h^2 \left(s\1_A,\hat{s}_{\hat{m}} \right) \right] \leq   \left|\restriction{\sqrt{s}}{A} \right|_{p,\sigma}^{\frac{2 d}{d + \sigma}} \left(\frac{\log^3 n}{n}  \right)^{\frac{ \sigma}{\sigma + d}} +  \frac{\log^3 n}{n} + \frac{R_n (\ell)}{n}
\end{eqnarray}
and by using  Theorem~2 of~\cite{Akakpo2012},
\begin{eqnarray*}
C \E \left[h^2 \left(s\1_A,\hat{s}_{\hat{m}} \right) \right]  \leq   \left|\restriction{\sqrt{s}}{A} \right|_{p,\sigma}^{\frac{2 d}{d + \sigma}} \left(\frac{\log n}{n} + 2^{-2 \ell d \theta} \right)^{\frac{ \sigma}{\sigma + d}} +  \frac{\log n}{n} + \frac{R_n (\ell)}{n}
\end{eqnarray*}
where $C > 0$ depends only on $\kappa$,$\sigma$,$d$,$p$ and  where  $$\theta = \frac{d +  \sigma}{\sigma} \left( \frac{\sigma}{d} - 2 \left(\frac{1}{p} - \frac{1}{2} \right)_+ \right).$$
If $\sigma > \sigma_1 (p,d)$ then $\theta > 1/2$. There exits thus $n_0$ (depending only on $\theta$), such that if $n \geq n_0$, $2^{-2 \ell \theta d} \leq \log n / n$, and hence
\begin{eqnarray*}
C' \E \left[h^2 \left(s\1_A,\hat{s}_{\hat{m}} \right) \right]\leq   \left|\restriction{\sqrt{s}}{A} \right|_{p,\sigma}^{\frac{2 d}{d + \sigma}} \left(\frac{\log n}{n}  \right)^{\frac{ \sigma}{\sigma + d}} +  \frac{\log n}{n} + \frac{R_n (\ell)}{n}.
\end{eqnarray*}
If $n \leq n_0$, we deduce from~{(\ref{equationvitessedeterministe})},
\begin{eqnarray*}
C \E \left[h^2 \left(s\1_A,\hat{s}_{\hat{m}} \right) \right]  &\leq&   \left|\restriction{\sqrt{s}}{A} \right|_{p,\sigma}^{\frac{2 d}{d + \sigma}} \left(\frac{\log^3 n_0}{n}  \right)^{\frac{ \sigma}{\sigma + d}} +  \frac{\log^3 n_0}{n} + \frac{R_n (\ell)}{n} \\
&\leq& C'' \left[ \left|\restriction{\sqrt{s}}{A} \right|_{p,\sigma}^{\frac{2 d}{d + \sigma}} \left(\frac{\log n}{n}  \right)^{\frac{ \sigma}{\sigma + d}} +  \frac{\log n}{n} + \frac{R_n (\ell)}{n} \right]
\end{eqnarray*}
where $C''$ depends only on $\sigma,d,p$. 
The conclusion ensues from the fact that $R_n (\ell)$ is upper-bounded by a constant depending only on $\kappa_0, b_1$.
\qed

\subsection{Proof of Proposition~\ref{ThmGeneral}.}~\label{ProofThemSelectGeneral}
We shall use the following lemma  whose proof is similar to the one of Lemma~\ref{LemmeProprietesDuTest}.
\begin{lemme} \label{LemmeProprietesDuTest2}
Set $\varepsilon = (2  + 3 \sqrt{2})/8$.
Under assumptions of Proposition~\ref{ThmGeneral}, there exists an universal constant $L_0 > 0$ such that for all  $L \geq L_0$ and $\xi > 0$,
\begin{eqnarray*} 
\forall f,f' \in S, \quad \left(1 - \varepsilon \right) H^2 \left(s\1_A, f'  \right) + T\left(f,f'  \right) \leq\left(1 + \varepsilon \right) H^2 \left(s\1_A, f \right) + L \frac{\Delta_S(f) + \Delta_S(f')}{n} + 22 \xi
\end{eqnarray*}
 with probability larger than $1 - e^{-n \xi}$.
\end{lemme}
\begin{proof}[Proof of Proposition~\ref{ThmGeneral}.]
By using the lemma above, with probability larger than $1 -  e^{-n \xi}$, for all $f \in S$,
\begin{eqnarray*}
\sup_{\substack{f' \in S}} \left\{\left(1 - \varepsilon \right) H^2 \left(s\1_A, f'  \right) + T\left(f,f'  \right)   -     L \frac{\Delta_S(f')}{n} \right\} \leq\left(1 + \varepsilon \right) H^2 \left(s\1_A, f \right) +   L \frac{\Delta_S(f)}{n}  + 22 \xi.
\end{eqnarray*}
Thus, if $ T(f,\hat{f} ) +   L  \frac{\Delta_S(f)}{n}-    L  \frac{\Delta_S(\hat{f})}{n} \geq 0$, 
\begin{eqnarray*}
\alpha H^2 (s\1_A, \hat{f} ) &\leq& (1 - \varepsilon) H^2 (s\1_A, \hat{f}) +   T (f,\hat{f} ) -   L  \frac{\Delta_S(\hat{f})}{n} +     L  \frac{\Delta_S(f)}{n} \\
&\leq& \left(1 + \varepsilon \right) H^2 \left(s\1_A,f \right)  +  2 L  \frac{\Delta_S(f)}{n}+  22 \xi.
\end{eqnarray*}
If $ T(f,\hat{f} ) +   L  \frac{\Delta_S(f)}{n}-    L  \frac{\Delta_S(\hat{f})}{n} < 0$, 
\begin{eqnarray*}
\alpha H^2 (f, \hat{f} ) &\leq& \alpha  H^2 (\hat{f},f ) +   T (\hat{f}, f)  - L  \frac{\Delta_S(f)}{n} +  L  \frac{\Delta_S(\hat{f})}{n}  \\
&\leq&  \sup_{f' \in S} \left\{\alpha H^2 ( \hat{f}, f') +   T(\hat{f}, f')  -     L  \frac{\Delta_S(f')}{n}  \right\}  +    L  \frac{\Delta_S(\hat{f})}{n} \\
&\leq& \wp (\hat{f}) \\
&\leq& \wp(f) + \frac{1}{n} \\
&\leq&  \sup_{f' \in S}  \left\{  \alpha H^2 \left(f,  f' \right) +  T \left(f, f'\right) -     L  \frac{\Delta_S(f')}{n}   \right\} + L \frac{\Delta_S(f)}{n} + \frac{1}{n}. 
\end{eqnarray*}
With $\upsilon = (1 - \varepsilon)/ \alpha - 1 > 0$,
\begin{eqnarray*}
\alpha H^2 (f, \hat{f} )  &\leq&  \left(1 + \upsilon^{-1} \right)  H^2 \left(f,  s \1_A \right)   \\ 
& &  \qquad + \sup_{ f' \in S}  \left\{ (1  - \varepsilon)   H^2 \left(s\1_A,  f' \right) +    T \left(f, f'\right) -    L \frac{\Delta_S(f')}{n}    \right\} + L \frac{\Delta_S(f)}{n} + \frac{1}{n}  \\
&\leq&    \left(1 + \upsilon^{-1} \right)  H^2 \left(f,  s \1_A \right)  +  \left[ \left(1 + \varepsilon \right) H^2 \left(s\1_A, f \right) +    L \frac{\Delta_S(f)}{n} +  22 \xi \right]  +  L \frac{\Delta_S(f)}{n}   +  \frac{1}{n} \\
&\leq&   \left(2 + \varepsilon + \upsilon^{-1} \right) H^2 \left(f,  s \1_A \right)   + 2 L \frac{\Delta_S(f)}{n}  +  22 \xi +  \frac{1}{n}.
\end{eqnarray*}
This leads to,
\begin{eqnarray*}
\alpha H^2 (s\1_A, \hat{f} )  &\leq& 2 \alpha  H^2 \left(s\1_A, f \right) + 2 \alpha H^2 (f, \hat{f}) \\
&\leq& 2 \left(2  + \alpha +  \varepsilon + \upsilon^{-1} \right)  H^2 \left(f,  s \1_A \right)  + 4 L \frac{\Delta_S(f)}{n} + 44 \xi  + \frac{2 }{n}.
\end{eqnarray*}
Finally, we have proved that there exists $C > 0$, such that,  with probability larger than $1 -   e^{-n \xi}$, for all $f \in S$,
\begin{eqnarray*}
C H^2 (s\1_A, \hat{f} )  \leq    H^2 \left(f,  s \1_A\right)  +  L \frac{\Delta_S(f)}{n} + \xi.
\end{eqnarray*}
The conclusion follows.
\end{proof}

\subsection{Proof of Corollary~\ref{CorVitesseAutoRegressif}.} \label{SectionPreuveOuDefinitNormeSup}
Throughout this proof, the  distance associated to the supremum norm $\|\cdot\|_{\infty}$ is denoted by $d_{\infty}$.
We shall use the following lemma (the first part may be deduced from the work of~\cite{Akakpo2012} whereas the second part may be deduced from results in~\cite{Dahmen1980}).
\begin{lemme} \label{CollectionV}
There exists a collection  $\W$ of (finite dimensional) linear spaces such that for all $p \in (0,+\infty]$,  $\beta >  (1/p - 1/2)_+$ and $f \in \mathscr{B}^{\beta} (\L^p ([0,1]))$, $L > 0$, $\tau > 0$, $\sigma > 0$,
$$C \inf_{W \in \W} \left\{ L^2 d_2^{2 \sigma} (g, W) + (\dim W) \tau \right\}    
\leq  \left(L |g|_{p,\beta}^{\sigma}\right)^{\frac{2}{2 \sigma \beta + 1}} \tau^{\frac{2 \sigma \beta }{2 \sigma \beta + 1}}  + \tau  $$
where $C > 0$ depends only on $p$, $\beta$.
Moreover, for all $\beta > 0$, $f \in \mathcal{H}^{\beta} ([0,1])$, $L > 0$, $\tau > 0$, $\sigma > 0$,
$$C' \inf_{W \in \W} \left\{ L^2 d_{\infty}^{2 \sigma} (g, W) + (\dim W) \tau \right\}    
\leq  \left(L |g|_{\infty,\beta}^{\sigma}\right)^{\frac{2}{2 \sigma \beta + 1}} \tau^{\frac{2 \sigma \beta }{2 \sigma \beta + 1}}  + \tau$$
where $C' > 0$ depends only on $\beta$.
\end{lemme}
Let us define 
$$u(x,y) = \frac{y - g(x)}{1 + \|g\|_{\infty}} \quad \text{and} \quad \Phi(x) = \phi \left(  (1+\|g\|_{\infty}) x  \right) \quad \text{for all $x,y \in [0,1]$.}$$
Let $\W$ be the family of linear spaces given by the lemma above. Define, for all $W \in \W$, the linear space
$$T_W = \{(x,y) \mapsto a (y - f(x)), a \in \R ,\, f \in W\}$$
and $\TT = \{T_W, \, W \in \W \}$. Since $\Phi$ belongs to  $\mathcal{H}^{\sigma} ([0,1])$, we deduce from Corollary 1 of~\cite{BaraudComposite} and  from our Theorem~\ref{ThmSelectGeneral} that there exists an estimator $\hat{s}$ such that
\begin{eqnarray*}
C \E \left[H^2 \left(s, \hat{s} \right) \right] \leq \inf_{T \in \TT} \left\{|\Phi|_{\infty,\sigma \wedge 1}^2 d_2^{2 (\sigma \wedge 1)} (u, T) + (\dim T) \tau_n  \right\} +  \inf_{W \in \W} \left\{ d_{\infty}^{2} (\Phi, W) + (\dim W)\frac{\log n}{n}\right\} 
\end{eqnarray*}
where $C > 0$ depends on $\sigma, \kappa$ and where
$$ \tau_n  = \left(\log n \vee \log \left(|\Phi|_{\infty,\sigma \wedge 1} \right) \right) \frac{\log n}{n}.$$
Now,
\begin{eqnarray*}
 \inf_{T \in \TT} \left\{|\Phi|_{\infty,\sigma \wedge 1}^2 d_2^{2 (\sigma \wedge 1)} (u, T) + (\dim T)  \tau_n \right\}  \leq \inf_{W \in \W} \left\{ |\phi|_{\infty,\sigma \wedge 1}^2 d_2^{2 (\sigma \wedge 1)} (g, W) + (\dim W + 1) \tau_n \right\}    
\end{eqnarray*}
and the conclusion follows from the lemma above. \qed

\subsection{Proof of Lemma~\ref{LemmeRegulariteAutoRegressif}.}
The first part of the lemma may be deduced from Proposition~4 of~\cite{BaraudComposite}.
For the second part, we shall build $\phi' \in \mathcal{H}^{\sigma} (\R)$ such that $\restriction{\phi'}{[0,1]} \not \in  \cup_{b > \sigma} \mathcal{H}^{b} ([0,1])$ and $g' \in \mathcal{H}^{\beta} ([0,1])$ such that $g'(0) = 0$ and
$$\phi' \circ g' \in  \mathcal{H}^{ \theta (\beta,\sigma)} ([0,1]) \setminus \cup_{b >  \theta (\beta,\sigma)} \mathcal{H}^{b}  ([0,1]).$$
By setting $\phi = \phi'$ and $g = -g'$, the function $f$ defined by
$$f(x,y) = \phi' \left(y - (-g' (x))\right) \quad \text{for all $x,y \in [0,1]$},$$
is suitable since $f(x,0) = \phi'\circ g' (x)$ and $f(0,y) = \phi'(y)$.

If $\sigma, \beta \leq 1$, we can choose $\phi'(x) = x^{\sigma}$ on $[0,1]$ and $g'(x) = x^{\beta}$. If  $\beta \geq \sigma \vee 1$, then choose $\phi' \in \mathcal{H}^{\sigma} (\R)$ such that $\restriction{\phi'}{[0,1]} \not \in  \cup_{b > \sigma} \mathcal{H}^{b}  ([0,1])$ and  $g'(x) = x$. If now, $\sigma \geq \beta \vee 1$, we choose  $\phi' \in \mathcal{H}^{\sigma} (\R)$ such that $\restriction{\phi'}{[0,1]} \not \in  \cup_{b > \sigma} \mathcal{H}^{b} ([0,1])$ and such that $\phi' (x) = x$ for all $x \in [0,1/2]$. We then consider $\zeta \in \mathcal{H}^{\beta} ([0,1]) \setminus \cup_{b >  \beta} \mathcal{H}^{b}   ([0,1])$  and $g'(x) =  (\zeta(x) - \zeta(0))/ (2 \sup_{y \in [0,1]} |\zeta(y) - \zeta(0)|)$. \qed

\subsection{Proof of Corollary~\ref{CorVitesseARCH}.}
Throughout this proof, $d_{\infty}$ stands for the distance associated to the supremum norm $\|\cdot\|_{\infty}$.
Let us define 
\begin{eqnarray*}
\forall x,y,z \in [0,1], \quad & & u(x,y) = (u_1(x,y),u_2(x,y),u_3(x,y)) =  \left(\frac{y - v_1 (x)}{1 + \|v_1\|_{\infty}}, \frac{v_2(x)}{ \|v_2\|_{\infty} }, \ \frac{v_3(x)}{ \|v_3\|_{\infty} }\right) \\
& &  \Phi(x,y,z) = \|v_3\|_{\infty} z \varphi\left((1 + \|v_1\|_{\infty}) \|v_2\|_{\infty}  x y \right).
\end{eqnarray*}
Let $\W$ be the family of linear spaces given by Lemma~\ref{CollectionV}. Define, for all $W \in \W$ the linear spaces
$$T_W = \left\{(x,y) \mapsto a (y - f(x)), \, a \in \R, \, f \in W \right\} \quad \text{and} \quad F_W = \left\{(x,y,z) \mapsto z f(x y), \, f \in W \right\}$$
and set  $\TT_1 = \{T_W, \, W \in \W\}$, $\TT_2 = \W$, $\TT_3 = \W$, $\mathbb{F} = \{F_W, \, W \in \W\}$.

It ensues from Corollary 1 of~\cite{BaraudComposite} and our Theorem~\ref{ThmSelectGeneral} that there exists an estimator $\hat{s}$ such that
\begin{eqnarray*}
C \E \left[H^2 \left(s, \hat{s} \right) \right] &\leq& \inf_{T \in \TT_1} \left\{\|v_3\|^2_{\infty} (1 + \|v_1\|_{\infty})^{2 (\sigma \wedge 1)} \|v_2\|_{\infty}^{2 (\sigma \wedge 1)}  |\varphi|_{\infty,\sigma}^2  d_2^{2 (\sigma \wedge 1)} (u_1, T) + (\dim T)   \tau_n^{(1)}  \right\} \\
& &  \!\!\!\! + \inf_{T \in \TT_2} \left\{\|v_3\|^2_{\infty} (1 + \|v_1\|_{\infty})^{2 (\sigma \wedge 1)}  \|v_2\|_{\infty}^{2 (\sigma \wedge 1)} |\varphi|_{\infty,\sigma}^2 d_2^{2 (1 \wedge \sigma)} (u_2, T) + (\dim T)   \tau_n^{(2)}  \right\}  \\
& &  \!\!\!\! + \inf_{T \in \TT_3} \left\{\|v_3\|_{\infty}^2 \|\varphi\|_{\infty}^2 d_2^{2} (u_3, T) + (\dim T)   \tau_n^{(3)}  \right\}  \\
& &  \!\!\!\! + \inf_{F\in \mathbb{F}} \left\{ d_{\infty}^{2} (\Phi, F) + (\dim F) \frac{\log n}{n} \right\}
\end{eqnarray*}
where
\begin{eqnarray*}
 \tau_n^{(1)}   &=& \left(\log n \vee \log \left(\|v_3\|^2_{\infty} (1 + \|v_1\|_{\infty})^{2 (\sigma \wedge 1)} |\varphi|_{\infty,\sigma}^2 \|v_2\|_{\infty}^{2 (\sigma \wedge 1)} \right)  \right) \frac{\log n}{n} \\
  \tau_n^{(2)}  &=& \left(\log n \vee \log \left(\|v_3\|^2_{\infty} (1 + \|v_1\|_{\infty})^{2 (\sigma \wedge 1)} |\varphi|_{\infty,\sigma}^2 \|v_2\|_{\infty}^{2 (\sigma \wedge 1)} \right)  \right) \frac{\log n}{n}  \\
    \tau_n^{(3)}  &=& \left(\log n \vee \log \left(\|v_3\|_{\infty}^2 \|\varphi\|_{\infty}^2 \right) \right) \frac{\log n}{n}.
\end{eqnarray*}
Hence,
\begin{eqnarray*}
C' \E \left[H^2 \left(s, \hat{s} \right) \right] &\leq&    \inf_{W \in \W} \left\{\|v_3\|^2_{\infty}|\varphi|_{\infty,\sigma}^2 \|v_2\|_{\infty}^{2 (\sigma \wedge 1)}   d_2^{2 (\sigma \wedge 1)} (v_1, W) + (\dim W)   \tau_n^{(1)}  \right\} \\
& & \quad  + \inf_{W \in \W} \left\{\|v_3\|^2_{\infty} (1 + \|v_1\|_{\infty})^{2 (\sigma \wedge 1)}   |\varphi|_{\infty,\sigma}^2  d_2^{2 (1 \wedge \sigma)} (v_2, W) + (\dim W)   \tau_n^{(2)}  \right\}  \\
& & \quad  + \inf_{W\in \mathbb{W}} \left\{\|\varphi\|_{\infty}^2  d_2^{2 (1 \wedge \sigma)} (v_3, W) + (\dim W)  \tau_n^{(3)}  \right\} \\
& & \quad  + \inf_{W\in \mathbb{W}} \left\{\|v_3\|_{\infty}^{2 } d_{\infty}^{2} (\varphi, W) + (\dim W) \frac{\log n}{n} \right\}.
\end{eqnarray*}
Calculating these minimums via Lemma~\ref{CollectionV} leads to the result. \qed

\subsection{Proof of Lemma~\ref{LemmeVitesseArch}.}
The first part of the lemma can be deduced from Proposition~4 of~\cite{BaraudComposite}.  For the second part, remark that, as in the proof of Lemma~\ref{LemmeRegulariteAutoRegressif} the problem amounts to finding
$\phi' \in \mathcal{H}^{\sigma}  (\R)$ with $\restriction{\phi'}{[0,1]} \not \in \cup_{a >  \sigma} \mathcal{H}^{a}  (\R)$, $v_i' \in  \mathcal{H}^{\beta_i}  ([0,1])$ for $i \in \{1,2\}$, $v_1'(0) = 0$, $v_2'(0) = 1$
such that
$$\sqrt{v'_2} \, \phi' (v_1' v_2') \in  \mathcal{H}^{ \theta (\beta_1,\beta_2,\sigma)}  ([0,1]) \setminus \bigcup_{b > \theta (\beta_1,\beta_2,\sigma)} \mathcal{H}^{b}   ([0,1]).$$

If $\theta (\beta_1,\beta_2,\sigma) = 2^{-1} (\beta_2 \wedge 1)$,   choose  $v_2'(x) = (1-x)^{1 \wedge \beta_2}$ and take $\phi'$ as being any function of $\mathcal{H}^{\sigma} (\R)$ such that $\restriction{\phi'}{[0,1]} \not \in \cup_{a >  \sigma} \mathcal{H}^{a}   (\R)$ and such that $\phi'(0) = 1$.
If $\theta (\beta_1,\beta_2,\sigma) = \sigma$, choose $v_1'(x) = 2 (\sqrt{1 + x} - 1)$, $v_2'(x) = 1/2 (\sqrt{1+x} + 1)$
and take $\phi'$ as being any function of $\mathcal{H}^{\sigma}  (\R)$ such that $\restriction{\phi'}{[0,1]} \not \in \cup_{a >  \sigma} \mathcal{H}^{a} (\R)$.
If  $\theta (\beta_1,\beta_2,\sigma) = \sigma \beta_1$, we may assume that $\sigma \leq 1$ and $\beta_1 \leq 1$. We can then choose $v_1'(x) = x^{\beta_1}$, $v_2'(x) = 1$ and $\phi'(x) =  x^{\sigma}$ for $x \in [0,1]$.
If  $\theta (\beta_1,\beta_2,\sigma) = \sigma \beta_2$, we may assume that $\sigma \leq 1$ and $\beta_2 \leq 1$ and choose $v_1'(x) = 1$ for $x \in [1/2,1]$, $v_2'(x) = 1 - (1-x)^{\beta_2}$ for $x \in [1/2,1]$ and $\phi'(x) =  (1-x)^{\sigma}$ for $x \in [0,1]$.
 Finally, if $\theta (\beta_1,\beta_2,\sigma) = \beta_1$, we may assume that $\beta_1 \leq 1$.  We can then choose $v_1'(x) = x^{\beta_1}$, $v_2'(x) = (1-x)^{1 \wedge \beta_2}$   and $\phi'$ such that $\phi'(x) = x$ for $x \in [0,1/2]$.
\qed

\subsection{Proof of Proposition~\ref{PropositionComplexity}.}
We proceed in 3 steps.
\begin{enumerate} [Step 1.]
\item We associate to each cube $K \in \cup_{m \in \M_{\ell}} m$, a place in the computer's memory. Then, for each $i \in \{1,\dots,n\}$ we determine the sets $K \in \cup_{m \in \M_{\ell}} m$  such that $ \1_{K} (X_i,X_{i+1}) > 0$.
There are at most $\ell$ such sets. This permits to store  all the  $\sum_{i=0}^{n-1} \1_{K} (X_i,X_{i+1})$ in  around $\mathcal{O} (n \ell d)$  operations.
Let for all $K \in \cup_{m \in \M_{\ell}} m$, $I_K$ and $J_K$ be the subsets of $[0,1]^d$ such that $K = I_K \times J_K$. 
We can store all the $\mu (J_K)$ in $\mathcal{O} (4^{\ell d})$ operations and all the  $\sum_{i=0}^{n-1} \1_{I_K} (X_i)$ in $\mathcal{O} (n \ell d)$ operations.
This  permits us to store quickly
$$\sum_{i=0}^{n-1} \1_{K} (X_i,X_{i+1}) \quad \text{and} \quad \sum_{i=0}^{n-1} \int_{[0,1]^{d}} \1_K (X_i,x) \d \mu (x)$$ 
for all $K \in \cup_{m \in \M_{\ell}} m$.  These values have to be calculated to know the $F_K (K')$ and thus to use the  algorithm presented in Section~\ref{SectionAlgorithm}. 
\item For each $K \in \cup_{m \in \M_{\ell}} m$, we use the algorithm of Section~\ref{SectionAlgorithm} to design $m'_K$. Let us denote by $j \in \{0,\dots,\ell\}$ the smallest integer such that $K \in \mathcal{K}_j$ where $\mathcal{K}_j$ is defined in Section~\ref{SectionChoixM}. 
\begin{itemize}
\item To find $m'_K$, we  begin to compute $\mathcal{E} (T^{\star} (K''))$ for all $K'' \in \cup_{m \in \M_{\ell} \setminus \M_{\ell-1}} m$ such that $K'' \cap K \neq \emptyset$.  The complexity of this is around the number of such sets, \emph{i.e},  $4^{(\ell - j) d}$.
\item Next, thanks to relation~{(\ref{RelationSomme})} we compute  $\mathcal{E} (T^{\star} (K''))$ for all $K'' \in \cup_{m \in \M_{\ell-1} \setminus \M_{\ell - 2}} m$ such that $K'' \cap K \neq \emptyset$.  There are $4^{(\ell - j-1) d}$ such sets. The complexity of this operation is thus $4^d \times 4^{(\ell - j-1) d}$.
\item By recurrence, we compute  $\mathcal{E} (T^{\star} (K''))$ for all $K'' \in \cup_{m \in \M_{\ell} \setminus \M_{j}} m$  such that $K'' \cap K \neq \emptyset$ in at most
$$ 4^{(\ell - j) d} +  4^d \times  \sum_{k=1}^{\ell - j  -1} 4^{k d} \leq 3 \times 4^{(\ell - j) d}$$
operations.
\item We get then $\mathcal{E} (T^{\star} ([0,1]^d))$ in $4^d j$ additional operations. 
\end{itemize}
We apply this algorithm for all $K \in \cup_{m \in \M_{\ell}} m$. When $K \in \mathcal{K}_j$, computing $m'_K$ requires  thus $\mathcal{O} \left(4^{(\ell - j) d} + 4^d j \right)$ operations. Since $|\mathcal{K}_j| = 4^{j d}$, computing all the $m'_K$ requires finally 
$$\sum_{j=0}^{\ell}  4^{j d} \left(4^{(\ell - j) d} + 4^d j \right) = \mathcal{O} \left(\ell 4^{(\ell + 1) d} \right) $$
operations.
\item Now, by slightly modifying the algorithm, we can compute~{(\ref{relationgamma2})} in $\mathcal{O} \left(4^{(\ell + 1) d}\right)$ operations. 
\end{enumerate}
\qed

\thanks{Acknowledgements: many thanks to Yannick Baraud for his suggestions, comments, careful reading of the paper.  We are thankful to Claire Lacour for sending us the source code of the procedure of~\cite{Akakpo2011}.}

\bibliographystyle{apalike}
 \bibliography{biblio}
\end{document}